\documentclass[hidelinks,onefignum,onetabnum]{siamart250211}

\definecolor{mygreen}{RGB}{55,126,34}
\usepackage{hyperref}
\hypersetup{colorlinks={true},linkcolor={mygreen},citecolor={mygreen}}

\usepackage{lipsum}
\usepackage{amsfonts}
\usepackage{graphicx}
\usepackage{epstopdf}
\usepackage{enumitem}
\usepackage{algorithm}
\usepackage{algorithmicx,algpseudocode}
\usepackage{adjustbox}
\usepackage{mathrsfs}
\ifpdf
  \DeclareGraphicsExtensions{.eps,.pdf,.png,.jpg}
\else
  \DeclareGraphicsExtensions{.eps}
\fi
\usepackage{colortbl}

\allowdisplaybreaks

\newsiamremark{remark}{Remark}
\newsiamthm{claim}{Claim}

\headers{RSSM for Optimization Over the Stiefel Manifold}{A. Y.-M. Cheung, J. Wang, M.-C. Yue, and A. M.-C. So}

\title{Randomized Submanifold Subgradient Method for Optimization Over the Stiefel Manifold\thanks{Submitted to the editors DATE.
}}

\author{Andy Yat-ming Cheung\thanks{Department of Applied Mathematics, The Hong Kong Polytechnic University
  (\texttt{andy.cheung@polyu.edu.hk}); Department of Systems Engineering and Engineering Management, The Chinese University of Hong Kong (\texttt{andycheungym.maths@gmail.com}).}
\and Jinxin Wang\thanks{Corresponding author. Department of Statistics, University of Chicago
  (\texttt{wangjinxin68@gmail.com}).}
\and Man-Chung Yue\thanks{Department of Applied Mathematics, The Hong Kong Polytechnic University
 (\texttt{manchung.yue@polyu.edu.hk}).}
\and Anthony~Man-Cho So\thanks{Department of Systems Engineering and Engineering Management, The Chinese University of Hong Kong 
  (\texttt{manchoso@se.cuhk.edu.hk}).}
}

\usepackage{amsopn}
\usepackage{tablefootnote}

\ifpdf
\hypersetup{
  pdftitle={Randomized Submanifold Subgradient Method for Optimization Over the Stiefel Manifold},
  pdfauthor={A. Y.M. Cheung, J. Wang, M.C. Yue, and A. M.-C. So}
}
\fi

\usepackage{booktabs}
\usepackage{makecell}
\usepackage{amssymb, bbding}
\usepackage{amsmath, enumitem}
\usepackage{tablefootnote}

\newcommand{\dist}{\mathrm{dist}}

\DeclareMathOperator*{\argmin}{argmin}

\def\R{\mathbb{R}}

\def\E{\mathbb{E}}

\def\eps{\varepsilon}

\def\pd{\partial}

\def\pd{\partial}

\def\Sym#1{\mathop{\mathrm{Sym}}(#1)}
\def\Skew#1{\mathop{\mathrm{Skew}}(#1)}

\def\Orth#1{\mathrm{O}(#1)}

\def\Stiefel#1#2{\mathrm{St}(#1,#2)}

\def\argmin{\mathop{\mathrm{argmin}}}

\def\Sym#1{\mathop{\mathrm{Sym}}(#1)}
\def\Skew#1{\mathop{\mathrm{Skew}}(#1)}

\def\Orth#1{\mathrm{O}(#1)}

\def\Stiefel#1#2{\mathrm{St}(#1,#2)}

\begin{document}

\maketitle

\begin{abstract}
Optimization over the Stiefel manifold is a fundamental computational problem in many scientific and engineering applications. Despite considerable research effort, high-dimensional optimization problems over the Stiefel manifold remain challenging, particularly when the objective function is nonsmooth. In this paper, we propose a novel coordinate-type algorithm, named \emph{randomized submanifold subgradient method} (RSSM), for minimizing a possibly nonsmooth weakly convex function over the Stiefel manifold and study its convergence behavior. In contrast to existing algorithms, RSSM operates on ``coordinate blocks'' defined by subsets of the columns of the matrix variable.
Similar to coordinate-type algorithms in the Euclidean setting, RSSM exhibits low per-iteration cost and is suitable for high-dimensional problems.
We prove that RSSM has an iteration complexity of {$\mathcal O(\ell^{2/3}\eps^{-4})$ for driving a natural stationarity measure below $\eps$ in expectation, where $\ell$ represents the number of blocks. We also establish the almost-sure convergence of RSSM and provide the almost-sure asymptotic convergence rate.}
To the best of our knowledge, this is the first convergence guarantee for coordinate-type algorithms for nonsmooth optimization over the Stiefel manifold. 
Lastly, we present numerical results on robust subspace recovery, orthogonal dictionary learning, and sparse principal component analysis to demonstrate the viability of our proposed method.
\end{abstract}



\section{Introduction}\label{sec: intro}
This paper focuses on the optimization problem
\begin{eqnarray} \label{mainpb}
\begin{array}{cl}
\min & f (X)\\
\text{s.t.} & X \in  \Stiefel {n} {p} ,
\end{array}
\end{eqnarray}
where $n,p\ge1$ are positive integers with $n \ge p$, $\Stiefel n p= \{ X \in \R^{n\times p}: X^\top X = I\} $ is the Stiefel manifold, and $f : \R^{n \times p} \to \R$ is a function that is $\tau$-weakly convex on some bounded convex neighborhood $\varOmega \subseteq \mathbb{R}^{n \times p}$ containing the Stiefel manifold $\Stiefel n p$ (i.e., the function $f(\cdot) + \tfrac{\tau}{2}\| \cdot \|^2$ is convex on $\varOmega$) for some $\tau \in \R$. Note that $f$ is not necessarily smooth.  Throughout the paper, unless otherwise specified, we assume that $\tau \ge 0$.

Problem~\eqref{mainpb} has attracted much attention from both optimization and machine learning communities due to its broad range of applications, including the Procrustes problem~\cite{gower2004procrustes, sato2013riemannian}, the joint diagonalization problem~\cite{theis2009soft}, Kohn-Sham total energy minimization~\cite{liu2015analysis}, robust subspace recovery~\cite{zhu2018dual}, fair principal component analysis~\cite{vu2022distributionally}, and orthogonal dictionary learning~\cite{li2021weakly}. For recent algorithmic developments and applications of optimization over the Stiefel manifold, we refer the reader to~\cite{absil2009optimization, boumal2023introduction, chen2024nonsmooth, absil2019collection,hosseini2013nonsmooth} and the references therein.

In modern applications, the dimension $n$ or $p$ of the Stiefel manifold could be very large. For example, as pointed out in \cite{li2019orthogonal}, deep neural networks attain an optimal generalization error when the weight matrix has orthonormal columns or rows. Constraining the columns or rows of the weight matrix to be orthonormal in the training of deep neural networks naturally gives rise to high-dimensional optimization problems over the Stiefel manifold. Unfortunately, despite considerable research effort, existing algorithms for optimization over the Stiefel manifold do not scale well and are only suitable for small- to medium-scale problems. When the problem has a nonsmooth objective function, it becomes even more challenging. The main goal of this paper is to address these challenges by developing an efficient algorithm with convergence guarantees for nonsmooth optimization over a high-dimensional Stiefel manifold.

In the Euclidean setting, coordinate-type algorithms constitute a classical approach to tackling high-dimensional optimization problems and have shown promising performance in many applications~\cite{nesterov2012efficiency,wright2015coordinate}. 
A natural idea for achieving our goal is therefore to extend coordinate-type algorithms to the manifold setting. This idea has been explored in a number of prior works, though they focus exclusively on smooth objective functions. Roughly speaking, the approaches in existing works can be divided into three groups. 

The first applies to manifolds that have a simple separable structure (e.g., a product manifold). Such a structure allows the development of coordinate-type algorithms that use simple operations to maintain the feasibility of the iterates; see, e.g., \cite{huang2021projection,huang2021riemannian,peng2023block,tian2021distributed}. However, as the Stiefel manifold does not have an obvious separable structure, it is difficult to directly apply the approaches in this group to tackle optimization problems over the Stiefel manifold. 


The second involves performing coordinate decomposition in the tangent spaces to the manifold. In \cite{gutman2022coordinate}, a general framework, called tangent subspace descent, is proposed for developing coordinate-type algorithms for manifold optimization. At each iteration, the tangent space to the manifold at the current iterate is decomposed into low-dimensional subspaces, and the next iterate is obtained by updating along a chosen subspace using the exponential map. The crux of this framework is the choice of the low-dimensional subspaces of the tangent space. As an instantiation of the framework, the work \cite{gutman2022coordinate} presents a coordinate-type algorithm for optimization over the Stiefel manifold, which modifies at most two columns per iteration. Such a column-wise update scheme can be seen as a generalization of the algorithms in \cite{jiang2022givens,massart2022coordinate,shalit2014coordinate} for optimization over the orthogonal manifold, which are derived based on Givens rotations. To further reduce the per-iteration computational cost, a row-wise update scheme is devised in~\cite{han2024riemannian} by another choice of tangent subspaces. The approach in \cite{han2024riemannian} is also suitable for other manifolds, including the Grassmannian manifold. However, all these methods either rely on the exponential map or modify at most two rows or columns. It is unclear whether coordinate-type algorithms that update more columns or rows at each iteration and/or use a low-cost retraction instead of an exponential map can be developed. Using a similar idea as in~\cite{gutman2022coordinate} with a Cholesky factorization-friendly basis for the tangent space, the work \cite{darmwal2023low} proposes a coordinate-type algorithm for optimization over the manifold of positive definite matrices. Although each iteration of this algorithm can update more than two columns or rows and avoid matrix exponentiation (which is required for the exponential map of the manifold of positive definite matrices), the efficiency of the update is guaranteed only when the objective function is smooth and takes a certain composite form (see~\cite[Eq.(2)]{darmwal2023low}).

The third consists of penalty approaches, which transform the original manifold optimization problem into an unconstrained one by introducing a penalty term for constraint violation in the objective function. The work~\cite{gao2019parallelizable} studies a penalty approach and develops coordinate-type algorithms for optimization over the Stiefel manifold. Nevertheless, the approach in~\cite{gao2019parallelizable} leads to infeasible-point algorithms, which might not be desirable in some applications. 

\begin{table}[t!]
\centering
\setlength\tabcolsep{3pt}
	\caption{Coordinate-type methods for optimization over the Stiefel manifold.}
    \resizebox{0.90\textwidth}{!}{
	\begin{tabular}{cccccc}
				\toprule
				{\bf Method } & {\bf Coordinate}& {\bf Objective} 
   &  {\bf Update} %
   & {\bf Feasible} 
   &  {\bf Retraction} \\
				\midrule
				\makecell[c]{\cite{gutman2022coordinate} \\ (TSD)} & intrinsic & \makecell[c]{Lipschitz \\Riemannian gradient} 
   & \makecell[c]{modify \\ at most \\ two columns}  %
   & \CheckmarkBold  
   & \makecell[c]{exponential \\ map} \\
				\midrule
    \makecell[c]{\cite{han2024riemannian}} & intrinsic & \makecell[c]{Lipschitz \\Riemannian gradient} 
   & \makecell[c]{modify \\ two rows} %
   & \CheckmarkBold  
   & \makecell[c]{exponential \\ map} \\
   \midrule
      \makecell[c]{\cite{gao2019parallelizable}\\ (PCAL)} & ambient & \makecell[c]{$C^2$ with bounded \\ Euclidean Hessian} & \makecell[c]{modify \\ multiple \\ columns} %
   & \XSolidBrush 
   & N/A
   \\
\midrule
   \makecell[c]{\cite{hanefficient}\\ (RSDM)} & intrinsic & \makecell[c]{bounded Euclidean \\ gradient and Hessian}  & \makecell[c]{left\\ orthogonal \\ action\tablefootnote{Equivalent to the addition of a low rank matrix; see~\cite[Section~4]{hanefficient}.}} %
   & \CheckmarkBold 
   & retraction
   \\ 
   \midrule\rowcolor{blue!5}\Gape[0pt][2pt]{\makecell[c]{Our work\\ (RSSM)}}~&~ambient & \Gape[0pt][2pt]{\makecell[c]{weakly convex,\\ nonsmooth}}
   & \Gape[0pt][2pt]{\makecell[c]{modify \\ multiple \\ columns}}
   & \Gape[0pt][2pt]{\CheckmarkBold}
   & \Gape[0pt][2pt]{\makecell[c]{polar decom-\\ position \eqref{eq-polarretr}}} \\
				\bottomrule
	\end{tabular}
    }
\label{tab:relatedworks}
\end{table}

The above discussion motivates the search for a coordinate-type algorithm that can tackle nonsmooth optimization problems over the Stiefel manifold and possesses more favorable properties than those in the aforementioned works. Our contributions can be summarized as follows.
\begin{itemize}
    \item We devise a new coordinate-type algorithm, called the \emph{randomized submanifold subgradient method} (RSSM), for minimizing a possibly nonsmooth weakly convex function over the Stiefel manifold (i.e., Problem~\eqref{mainpb}). A major novelty of RSSM lies in the way the ``coordinate blocks'' are defined: Instead of performing coordinate updates with respect to the intrinsic coordinates by decomposing the tangent spaces to the manifold as in~\cite{gutman2022coordinate,han2024riemannian}, RSSM performs coordinate updates with respect to the ambient Euclidean coordinates by decomposing the Stiefel manifold into submanifold blocks and taking a retracted partial Riemannian subgradient step along a randomly chosen submanifold block.

    \item We show that the said submanifold blocks are $1$-proximally smooth and that the corresponding projections can be computed efficiently in closed form. RSSM therefore exhibits a low per-iteration computational cost and is particularly suitable for high-dimensional instances of Problem~\eqref{mainpb}.
    
    \item We prove that RSSM has an iteration complexity of {$\mathcal O(\ell^{2/3}\eps^{-4})$} for driving a natural stationarity measure of Problem~\eqref{mainpb} below $\eps>0$ in expectation, where $\ell \ge 2$ represents the number of blocks. We also establish the almost-sure convergence of RSSM and provide the almost-sure asymptotic convergence rate.

    \item To facilitate our analysis, we develop two theoretical tools. The first is a new Riemannian subgradient inequality for weakly convex functions on proximally smooth matrix manifolds, which could be of independent interest in other matrix optimization studies. The second is a novel theoretical construct called the \emph{averaging operator}, which is a positive definite operator on $\R^{n\times p}$ and permeates the proofs via the adaptive  metric it induces on the ambient Euclidean space. Remarkably, these two tools together allow us to prove the sufficient decrease of our proposed method.
\end{itemize}

{We close this section with two remarks.} First, let us mention the recent paper~\cite{hanefficient}, which presents another submanifold-based coordinate-type algorithm called RSDM for optimization over the Stiefel manifold and appears around the same time as the preliminary version of this paper~\cite{cheung2024randomized}. Unlike RSSM, which operates on submanifolds that are defined in the ambient Euclidean coordinates and can handle nonsmooth objective functions, RSDM operates on submanifolds that are defined in intrinsic coordinates and applies only to smooth objective functions. For a comparison of various coordinate-type algorithms for optimization over the Stiefel manifold, see Table~\ref{tab:relatedworks}. {Second, similar to other coordinate-type methods, our framework is most effective when the objective function has a block separable structure, so that the cost of computing a partial Euclidean subgradient is proportional to the block size.
Some examples of such an objective function that arise in applications can be found in Section~\ref{sect-exp}.
}

\section{Notation and Preliminaries} \label{sect-prelim}

Let $\xi, \eta\in \R^{n\times p}$ be arbitrary. The \emph{Frobenius inner product} between $\xi$ and $\eta$ is denoted by $\langle\,\xi, \eta\,\rangle = \mathrm{tr}(\xi^\top \eta)$, while the \emph{Frobenius norm} of $\xi$ is denoted by $\|\xi\| = \sqrt{\langle \xi, \xi\rangle}$. The \emph{operator norm} and \emph{nuclear norm} of $\xi$ are denoted by $\|\xi\|_{\mathsf{op}} = \sup_{v\in \R ^p,\, \|v\|=1} \|\xi v\|$ and $\|\xi\|_\ast = \mathrm{tr}((\xi^\top \xi)^{1/2})$, respectively.
Given a self-adjoint positive definite linear operator $\mathcal D : \R^{n \times p} \to \R^{n \times p}$, the \emph{Mahalanobis inner product} between $\xi$ and $\eta$ is denoted by $\langle \xi, \eta\rangle_{\mathcal D} = \langle \mathcal D(\xi), \eta\rangle$, while the \emph{Mahalanobis norm} of $\xi$ is denoted by $\|\xi\|_{\mathcal D} = \sqrt{\langle \mathcal D(\xi), \xi\rangle}$. 
When $n=p$, the \emph{symmetric part} and \emph{skew-symmetric part} of $\xi$ are denoted by $\Sym{\xi} = \tfrac{1}{2}(\xi+\xi^\top)$ and $\Skew{\xi}=\tfrac{1}{2}(\xi-\xi^\top)$, respectively.
The symbols $I$, $J$, and $0$ represent the identity, all-one, and zero matrices, respectively. For a subset $C \subseteq [p] = \{1, \ldots, p\}$ and a matrix $X \in \R^{n \times p}$, the cardinality of $C$ is denoted by $|C|$, and the submatrix obtained by extracting all columns of $X$ corresponding to the indices in $C$ (resp., in $[p] \setminus C$) is denoted by $X_C \in \R^{n \times |C|}$ (resp., by $X_{-C} \in \R^{n \times (p-|C|)}$). For $\ell \geq 2$, the collection of $2$-sets or unordered pairs in $[\ell]$ is denoted by $\binom{[\ell]}{2} = \left\{\,\{i,j\} : i, j \in [\ell]\,\text{and}\, i \neq j\,\right\}$. Given a closed subset $\mathcal S \subseteq \R^{n \times p}$ and a point $X \in \R^{n \times p}$, the \emph{distance} from $X\in \R^{n\times p}$ to $\mathcal S$ is denoted by $\dist(X, \mathcal S) = \inf_{Y \in \mathcal S} \|X-Y\|$, while the projection of $X$ onto $\mathcal{S}$ is denoted by $\mathcal P_{\mathcal S}(X) = \{ Y \in \mathcal{S} : \|X-Y\| = \dist(X,\mathcal{S}) \}$.

\subsection{Weakly Convex Optimization on Stiefel Manifold} \label{subsect-weakcvx}

Given a $\tau$-weakly convex function $h:\R^{n\times p} \rightarrow \R$, which is necessarily Clarke regular~\cite[Proposition 4.5]{vial1983strong}, we have $h(\cdot) = g(\cdot) - \tfrac{\tau}{2}\|\cdot\|^2$ for some convex function $g:\R^{n \times p} \rightarrow \R$. The Clarke subdifferential of $h$ at $X \in \R^{n \times p}$ is given by $\pd h(X) = \pd g(X) - \tau X$, where $\pd g(X)$ is the usual convex subdifferential of $g$ at $X \in \R^{n \times p}$; see \cite[Proposition 4.6]{vial1983strong}. Moreover, we have
\begin{eqnarray} \label{eq-weakcvx}
h(Y) \geq h(X) + \langle\,\widetilde \nabla h(X), Y-X\,\rangle - \tau\|Y - X\|^2/2
\end{eqnarray}
for all $X, Y \in \R^{n \times p}$ and $\widetilde \nabla h(X) \in \pd h(X)$; see~\cite[Proposition 4.8]{vial1983strong}.

Let $\mathcal M \subseteq \R^{n \times p}$ be a smooth matrix manifold and $T_X \mathcal M$ be the tangent space to $\mathcal{M}$ at a point $X \in \mathcal M$. For a weakly convex (and hence Clarke regular) function $h$ in the ambient Euclidean space $\R^{n \times p}$, the Riemannian (Clarke) subdifferential of $h$ at $X \in \mathcal M$ is given by 
\begin{eqnarray} \label{eq-Rsubdiff}
\pd_{\mathcal M} h(X) = \mathcal P_{T_X \mathcal M} (\pd h(X)) = \{\,\mathcal P_{T_X \mathcal M}(\widetilde \nabla h(X))\,:\,\widetilde \nabla h(X) \in \pd h(X)\,\};
\end{eqnarray}
see \cite[Theorem 5.1]{yang2014optimality}. 
We call an $X \in \mathcal M$ that satisfies $0 \in \pd_{\mathcal M} h(X)$ a \emph{stationary point} of $h$ over $\mathcal M$. 
Given an Euclidean subgradient $\widetilde \nabla h(X) \in \pd h(X)$ of $h$ at $X \in \mathcal{M}$, the corresponding Riemannian subgradient is denoted by $\widetilde \nabla_{\mathcal M} h(X) = \mathcal P_{T_X \mathcal M}(\widetilde \nabla h(X))$. 
In particular, for the Stiefel manifold, it follows from~\cite[Example 3.6.2]{absil2009optimization} that $T_X \Stiefel n p = \{\,\xi \in \R^{n \times p}\,:\,\xi^\top X + X^\top \xi = 0\,\}$ and $\mathcal P_{T_X \Stiefel n p}(\xi) = \xi - X\Sym{X^\top \xi}$.

Many iterative algorithms for optimization over the manifold $\mathcal M$ update its iterate $X$ by first moving on the tangent space $T_X \mathcal M$ along a given direction and then mapping the resulting point back onto the manifold. The latter can be achieved using the exponential map (see, e.g.,~\cite[Section 5.4]{absil2009optimization}), though it can be numerically challenging to compute. This motivates the use of a \emph{retraction}, which is a locally smooth map from the tangent bundle of $\mathcal{M}$ to $\mathcal M$ that approximates the exponential map up to the first order; see \cite[Definition 1]{absil2012projection}. One important example of a retraction on $\mathcal{M}$ is the \emph{projective retraction}, which is given by $(X,\xi) \mapsto \mathcal{P}_{\mathcal{M}}(X+\xi)$ for $X \in \mathcal{M}$ and $\xi \in T_X \mathcal{M}$~\cite[Proposition 5]{absil2012projection}. For the Stiefel manifold $\Stiefel{n}{p}$, we focus on the polar decomposition-based retraction
\begin{eqnarray} \label{eq-polarretr}
\mathrm{Retr}_X(\xi) = (X + \xi)(I + \xi^\top \xi)^{-1/2}
\end{eqnarray}
for $X \in \Stiefel n p$ and $\xi \in T_X \Stiefel n p${. We discuss the possibility of using other retractions in Remark~\ref{remark: other_retraction}}. It is known that the polar decomposition-based retraction~\eqref{eq-polarretr} coincides with the projective retraction on $\Stiefel{n}{p}$, and the projection $\mathcal{P}_{\Stiefel{n}{p}}$ satisfies the Lipschitz-like property
\begin{eqnarray} \label{eq-lipslike}
\|\,\mathrm{Retr}_X(\xi) - Y\,\| = \|\,\mathcal P_{\Stiefel n p}(X + \xi) - \mathcal P_{\Stiefel n p}(Y)\,\| \leq \|\,X+\xi - Y\,\|
\end{eqnarray}
for all $X, Y \in \Stiefel n p$ and $\xi \in T_X \Stiefel n p$; see~\cite[Lemma 1]{li2021weakly}. Moreover, the map $\xi \mapsto \mathcal P_{\Stiefel n p}(X+\xi)$ satisfies the second-order boundedness condition
\begin{align} \label{eq-secondbdd}
    \|\,\mathcal P_{\Stiefel n p}(X + \xi) - X - \xi\,\| \leq \|\xi\|^2/2
\end{align}
for all $X \in \Stiefel n p$ and $\xi \in T_X \Stiefel n p$.\footnote{{It was shown in~\cite[Appendix~E.1]{liu2019quadratic} that
$\| \mathcal P_{\Stiefel n p}(X + \xi) - (X+\xi) \|
\le \|\xi\|^2$
holds for any $X \in \Stiefel n p$ and $\xi \in T_X \Stiefel n p$ satisfying
$\|\xi\| \le 1$.
By using the identity $\| \mathcal P_{\Stiefel n p}(X + \xi) - (X+\xi) \|^2
= \| (I + \xi^\top \xi)^{1/2} - I \|^2$ and a spectral decomposition of $\xi^\top \xi$, we get the sharper bound~\eqref{eq-secondbdd}.
}}

\subsection{Proximal Smoothness}  \label{subsect-proxsmooth} 

The following notion will be crucial as we deal with submanifolds of the Stiefel manifold in our subsequent development.
\begin{definition}[{\cite[Theorems 4.1 and 4.11]{clarke1995proximal}}]
    A closed set $\mathcal S \subseteq \R^{n\times p}$ is $R$-proximally smooth if the projection $\mathcal P_{\mathcal S}(X)$ is a singleton whenever $\dist(X, \mathcal S) < R$.
\end{definition}

It is known that the Stiefel manifold $\Stiefel n p$ is $1$-proximally smooth for any integers $n,p \ge 1$ with $n \ge p$; see, e.g.,~\cite[Proposition 7]{absil2012projection}.
Moreover, a smooth manifold $\mathcal M$ is $R$-proximally smooth if and only if 
\begin{eqnarray} \label{eq-proxsmooth}
2 R \, \langle\,\zeta, X' - X\,\rangle \leq \|\zeta\| \cdot \|X' - X\|^2
\end{eqnarray}
for all $X, X' \in \mathcal M$ and $\zeta \in N_X{\mathcal M}$, where $N_X \mathcal M$ denotes the \emph{normal space} to $\mathcal M$ at $X$. This can be deduced, e.g., by combining the results in~\cite[Theorem 2.2, Proposition 2.3]{adly2016preservation} and~\cite[Example 6.8]{rockafellar2009variational}.



\section{Randomized Submanifold Subgradient Method} \label{sect-RSSM}

In this section, we introduce RSSM, a coordinate-type method for solving Problem \eqref{mainpb}. Central to the development of RSSM is the notion of a \emph{submanifold block}, which is a submanifold of the Stiefel manifold induced by a collection of orthonormal columns.
Each iteration of RSSM performs a projected Riemannian subgradient step on the submanifold block induced by randomly selected columns of the current iterate. As we shall see, submanifold blocks enjoy many desirable properties, which facilitate the easy computation of the projected Riemannian subgradient step.

\subsection{Submanifold Block} \label{subsect-submfdblk}

To develop a coordinate-type algorithm for optimization over the Stiefel manifold, an intuitive approach is to update a subset of columns at each iteration. One fundamental challenge of such an approach is to maintain the feasibility of an iterate after updating the columns. This can be tackled by restricting the update to the subspace that is orthogonal to the span of the unaltered columns. We are thus led to the following definition.

\begin{definition}[Submanifold Block] \label{def-submfdblk}
Let $C \subseteq [p]$ be fixed. The \emph{submanifold block at $X \in \Stiefel n p$ with respect to $C$} is defined as
\begin{align} \label{eq-submfdblk}
    {\mathcal M_{X_{-C}}} = \{\,Y \in \Stiefel n {|C|}\,:\,{X_{-C}^\top} Y = 0\,\}.
\end{align}
\end{definition}
By definition, the submanifold block $\mathcal M_{X_{- C}}$ can be written as the intersection $\mathcal M_{X_{- C}} = \Stiefel n {|C|}\cap \{Y \in \R^{n \times |C|}\,:\,X_{- C}^\top Y = 0\}$. The intersection is transversal, since $T_{Z} \Stiefel n {|C|} + T_{Z} \{Y \in \R^{n \times |C|}\,:\,X_{- C}^\top Y = 0\} = \R^{n \times |C|}$ for any $Z \in \mathcal M_{X_{- C}}$. Therefore, by \cite[p.30]{guillemin2010differential}, the submanifold block $\mathcal M_{X_{- C}}$ is itself a smooth submanifold of the Stiefel manifold $\Stiefel{n}{|C|}$.
The tangent space to $\mathcal M_{X_{- C}}$ at $X_C \in \mathcal M_{X_{- C}}$ is 
\begin{align} \label{eq-TSsubmfdblk}
    T_{X_C}\mathcal M_{X_{- C}} = \{\,\xi \in T_{X_C}\Stiefel n {|C|}\,:\,X_{- C}^\top \xi = 0\,\}.
\end{align}

As mentioned above, our proposed RSSM performs a projected Riemannian subgradient step on a submanifold block at each iteration. The following lemma guarantees that the projection is well-defined, despite the non-convexity of the submanifold block, and provides a closed-form formula for the projection. In particular, when the search direction lies in the tangent space to the submanifold block, the projection reduces to the projection onto a lower-dimensional Stiefel manifold, which is readily computable via~\eqref{eq-polarretr}.

\begin{lemma}[Proximal Smoothness and Projection onto Submanifold Block] \label{lem-projsubmfdblk}
For any $C \subseteq [p]$ and $X \in \Stiefel n p$, the submanifold block $\mathcal M_{X_{-C}}$ is $1$-proximally smooth. In particular, for any $\Xi \in \R^{n \times |C|}$ with $\dist(\Xi, \mathcal M_{X_{- C}}) < 1$, we have 
    \begin{align*}
        \mathcal P_{\mathcal M_{X_{- C}}}(\Xi) &= \mathcal P_{\Stiefel n {|C|}}\left(\,\left(I - X_{- C}X_{- C}^\top \right)\Xi\,\right)\\ 
        &= \left(I - X_{- C}X_{- C}^\top \right)\Xi \left[\,\Xi^\top \left(I - X_{- C}X_{- C}^\top \right)\Xi\,\right]^{-\frac{1}{2}}.
    \end{align*}
    If in addition $X_{- C}^\top \Xi = 0$, then $\mathcal P_{\mathcal M_{X_{- C}}}(\Xi) = \mathcal P_{\Stiefel n {|C|}}(\Xi)$.
\end{lemma}

\begin{proof}
	Let $\Xi \in \R^{n \times |C|}$ be such that $\dist(\Xi, \mathcal M_{X_{- C}}) < 1$. Consider the decomposition $\Xi = X_{- C} X_{- C}^\top \Xi + (I - X_{- C}X_{- C}^\top) \Xi$. For any $U \in \mathcal M_{X_{- C}}$, we have
	\begin{align*}
		\|\Xi - U\|^2 &= \left\|X_{- C} X_{- C}^\top \Xi + (I - X_{- C}X_{- C}^\top) \Xi - U\right\|^2\\
		&= \left\|X_{- C} X_{- C}^\top \Xi\right\|^2 + \left\|(I - X_{- C}X_{- C}^\top) \Xi - U\right\|^2.
	\end{align*}
    Since $\dist(\Xi,\Stiefel n {|C|}) \leq \dist(\Xi, \mathcal M_{X_{- C}}) < 1$ and $\Stiefel n {|C|}$ is $1$-proximally smooth, we obtain
    \begin{align*}	
        &\argmin_{U \in \Stiefel n {|C|}} \left\|(I - X_{- C}X_{- C}^\top)\Xi - U\right\|^2 = \mathcal P_{\Stiefel n {|C|}}\left(\,(I - X_{- C}X_{- C}^\top)\Xi\,\right) \\
        &= (I - X_{- C}X_{- C}^\top)\Xi \left[\,\Xi^\top (I - X_{- C}X_{- C}^\top)\Xi\,\right]^{-1/2}\in \mathcal M_{X_{- C}} \subseteq \Stiefel n {|C|}.
    \end{align*}
    Here, the second equality follows from the fact that $\mathcal P_{\Stiefel n {|C|}}(\widetilde\Xi) = \widetilde\Xi(\widetilde\Xi^\top \widetilde\Xi)^{-1/2}$ when $\dist(\widetilde\Xi, \Stiefel n {|C|}) < 1$ and that 
    \begin{align*}
    \dist\left( (I - X_{- C}X_{- C}^\top)\Xi,\Stiefel n {|C|} \right) &\leq \dist\left( (I - X_{- C}X_{- C}^\top)\Xi, \mathcal M_{X_{- C}} \right)\\ 
    &\leq \dist(\Xi, \mathcal M_{X_{- C}}) < 1.
    \end{align*}
    Therefore, the projection $\mathcal P_{\mathcal M_{X_{- C}}}(\Xi)$ is uniquely given by
    \begin{align*}
        \mathcal P_{\mathcal M_{X_{- C}}}(\Xi) &= \argmin_{U \in \mathcal M_{X_{- C}}} \|\Xi - U\|^2 = \argmin_{U \in \mathcal M_{X_{- C}}} \left\|(I - X_{-C}X_{- C}^\top)\Xi - U\right\|^2\\
        &= (I - X_{- C}X_{- C}^\top)\Xi \left[\,\Xi^\top (I - X_{- C}X_{- C}^\top)\Xi\,\right]^{-1/2}.
    \end{align*}
    It follows that $\mathcal M_{X_{- C}}$ is $1$-proximally smooth. Furthermore, if $X_{- C}^\top \Xi = 0$, then the above formula gives $\mathcal P_{\mathcal M_{X_{- C}}}(\Xi) = \Xi (\Xi^\top \Xi)^{-1/2} = \mathcal P_{\Stiefel n {|C|}}(\Xi)$.
\end{proof}

\subsection{Partial Riemannian Subgradient Oracle} \label{subsect-partial}

Recall that the objective function $f$ in Problem~\eqref{mainpb} is $\tau$-weakly convex on some bounded convex neighborhood $\varOmega \subseteq \R^{n \times p}$ containing $\Stiefel n p$. By~\cite[Proposition 4.4]{vial1983strong}, we know that $f$ is $L$-Lipschitz continuous on $\varOmega$ for some $L>0$. To define the updates in our proposed RSSM, we need the subgradient of the restriction of $f$ to a submanifold block. Towards that end, let $X \in \R^{n \times p}$ and $C \subseteq [p]$ be fixed. Consider the map $\Phi_{X,C} : \R^{n\times |C|} \to \R^{n\times p}$ given by $\Phi_{X,C}(Y) = YI_C^\top + X_{- C}I_{- C}^\top$. We define the \emph{partial Euclidean (Clarke) subdifferential} of $f$ with respect to $C$ at $X$ as 
\[\pd^{C}f(X) = \pd (f \circ \Phi_{X,C})(X_C) \subseteq \R^{n \times |C|}.\]
A \emph{partial Euclidean (Clarke) subgradient} of $f$ with respect to $C$ at $X$ is denoted by 
\[\widetilde \nabla^{C} f(X) = \widetilde \nabla (f \circ \Phi_{X,C})(X_C) \in \pd^{C} f(X).\] The following proposition reveals the relationship between the partial and full Euclidean (Clarke) subdifferentials.

\begin{proposition}[Relationship between Partial and Full Euclidean Subdifferentials] \label{prop-structural}
For any $X \in \R^{n \times p}$ and $C \subseteq [p]$, we have $\pd^{C}\!f(X) = \pd f(X)I_C$, i.e., for any $\widetilde \nabla^{C}\!f(X) \in \pd^{C}\! f(X)$, there exists $\widetilde \nabla f(X) \in \pd f(X)$ such that $\widetilde \nabla^{C}\! f(X) = \widetilde \nabla f(X)I_C$.
\end{proposition}

\begin{proof}
The function $f$, being weakly convex, is Clarke regular~\cite[Proposition 4.5]{vial1983strong}. Hence, by the subdifferential chain rule~\cite[Theorem 10.6]{rockafellar2009variational} and the fact that $\Phi_{X,C}(X_C)=X$, we have $\pd^{C}\! f(X) = \pd(f \circ \Phi_{X,C})(X_C) = \nabla \Phi_{X,C}(X_C)^\ast \pd f(X) = \pd f (X) I_C$, where $\nabla \Phi_{X,C}(X_C)^\ast$ is the adjoint of the Jacobian of $\Phi_{X,C}$ at $X_C$.
\end{proof}

The index set $C$ induces the submanifold block $\mathcal M_{X_{- C}}$ in~\eqref{eq-submfdblk}. We define the \emph{partial Riemannian (Clarke) subdifferential} of $f$ with respect to $C$ at $X \in \Stiefel{n}{p}$ as 
\[ \pd^C_{\mathcal M_{X_{- C}}} f(X) = \pd_{\mathcal M_{X_{- C}}}(f \circ \Phi_{X,C})(X_C). \]
A \emph{partial Riemannian (Clarke) subgradient} of $f$ with respect to $C$ at $X$ is denoted by
\begin{align} 
\widetilde \nabla^C_{\mathcal M_{X_{- C}}} f(X) &= \widetilde \nabla_{\mathcal M_{X_{- C}}} (f \circ \Phi_{X,C})(X_C) \nonumber \\
&= \mathcal P_{T_{X_C}\mathcal M_{X_{- C}}} \left(\widetilde \nabla^{C} f(X) \right) \in \pd^C_{\mathcal M_{X_{- C}}} f(X), \label{eq:parRiesub}
\end{align}
where $\widetilde\nabla^{C} f(X) \in \pd^{C} f(X)$. It can be verified that $\mathcal P_{T_{X_C} \mathcal M_{X_{- C}}}(\xi) = X_C \Skew{X_C^\top \xi} + (I-XX^\top)\xi$ for all $\xi \in \R^{n \times |C|}$.

Given an integer $\ell \geq 2$, let $\mathfrak C = \{\,C_1, \ldots, C_\ell\,\}$ be a partition of $[p]$, i.e., $C_i \cap C_j = \varnothing$ and $\bigcup_{i=1}^\ell C_i = [p]$. For notational simplicity, given a matrix $X \in \Stiefel{n}{p}$, 
we denote $p_i = |C_i|$, $X_i = X_{C_i}$, $X_{-i} = X_{- C_i}$ for $i \in [\ell]$; and $C_{ij} = C_i \cup C_j$, $p_{ij} = |C_{ij}|$, $X_{ij} = X_{C_{ij}}$, $X_{-ij} = X_{- C_{ij}}$, $\mathcal M_{X_{-ij}} = \{\,Y \in \Stiefel n {p_{ij}}\,:\,X_{-ij}^\top Y = 0\,\} \subseteq \Stiefel n {p_{ij}}$ for $\{i,j\} \in \binom{[\ell]}{2}$. {We summarize the frequently used symbols in Table~\ref{tab: notation} for easy reference.} 

We now present our proposed RSSM in Algorithm~\ref{alg-RSSM}.

\begin{table}[t]
{
\begin{center}
\caption{{Frequently used symbols.}}\label{tab: notation}
\hspace{-0.3cm}
\begin{adjustbox}{max width=\linewidth}
\begin{tabular}{@{}l|l@{}} 
\toprule
{\bf Symbol} & {\bf Explanation}  \\ \cmidrule(r){1-2}
  $X_{C} \in \mathbb{R}^{n \times |C|}; X_{-C} \in \mathbb{R}^{n \times (p-|C|)}$       & Submatrix of $X$ with columns indexed by $C$; by $[p]\setminus C$  \\ 
  \cmidrule(r){1-2}
$\widetilde{\nabla}^C_{\mathcal{M}_{X_{-C}} }f(X)\in \mathbb{R}^{n \times |C|}$       &  \makecell[l]{A partial Riemannian subgradient of $f$ with respect to \\ $C$ at $X \in \Stiefel{n}{p}$}            \\ 
\cmidrule(r){1-2}
  $X_{ij} \in \mathbb{R}^{n \times p_{ij}}; X_{-ij}\in \mathbb{R}^{n \times (p-p_{ij})}$ & Submatrix of $X$ with columns indexed by $C_{ij}$; by $[p]\setminus C_{ij}$    \\ 
  \cmidrule(r){1-2}
   $\mathcal{M}_{X_{-C}}\subseteq \Stiefel n {|C|}; \mathcal{M}_{X_{-ij}}\subseteq \Stiefel n {p_{ij}}$       & \makecell[l]{Submanifold block at $X$ with respect to $C$; with respect \\ to $C_{ij}$}  \\ 
   \cmidrule(r){1-2}
  $\widetilde{\nabla}^C f(X); \widetilde{\nabla}^{C_{ij}}f(X)$ & \makecell[l]{A partial Euclidean subgradient of $f$ with respect to $C$ \\ at $X$; with respect to $C_{ij}$ at $X$}  \\ 
  \cmidrule(r){1-2}
    $\widetilde{\nabla}^{C_{ij}}_{\mathcal{M}_{X_{-ij}}}f(X) \in \mathbb{R}^{n \times p_{ij}}$  &  \makecell[l]{A partial Riemannian subgradient of $f$ with respect to $C_{ij}$ \\ at $X \in \Stiefel{n}{p}$}          \\ 
    \bottomrule
\end{tabular}
\end{adjustbox}
\end{center}}
\end{table}



\begin{algorithm}[ht]
\caption{Randomized Submanifold Subgradient Method for Problem~\eqref{mainpb}} \label{alg-RSSM}
\textbf{Input}: An initial point $X^0 \in \Stiefel n p$, a partition $\mathfrak C = \{\,C_1, \ldots, C_\ell\,\}$ of $[p]$ with $\ell \geq 2$, and a sequence of stepsizes $\{\gamma_k\}_k$ {(e.g., as specified in Theorems \ref{thm-expectation} and \ref{thm-convrate})}.
\begin{algorithmic}[1]
\For{$k = 0, 1, 2, \ldots$}
\State Sample $\{i, j\} \sim \mathrm{Uniform}\binom{[\ell]}{2}$ and compute the partial Riemannian subgradient\begin{equation} \label{eq-partialRiemsubgrad}
\widetilde \nabla^{C_{ij}}_{\mathcal{M}_{X_{-ij}^k}} f(X^k) = \mathcal P_{T_{X_{ij}^k} \mathcal M_{X_{-ij}^k}}\left(\widetilde\nabla^{C_{ij}} f(X^k)\right) \in T_{X_{ij}^k} \mathcal M_{X_{-ij}^k}.
\end{equation}
\State \label{line:update} Perform the updates
\begin{align}
X^{k+1}_{ij} &= \mathcal P_{\mathcal M_{X_{-ij}^k}} \left( X_{ij}^k - \gamma_k \widetilde \nabla^{C_{ij}}_{\mathcal{M}_{X_{-ij}^k}} f(X^k) \right) \nonumber \\
&= \mathcal P_{\Stiefel n {p_{ij}}} \left( X_{ij}^k - \gamma_k \widetilde \nabla^{C_{ij}}_{\mathcal{M}_{X_{-ij}^k}} f(X^k) \right), \label{eq-projsubmfdblk} \\ 
X_{-ij}^{k+1} &= X_{-ij}^k. \nonumber
\end{align}
\EndFor
\end{algorithmic}
\end{algorithm}
At each iteration $k$, RSSM first randomly selects an unordered column block index pair $\{i,j\} \in \binom{[\ell]}{2}$ uniformly. Then, it computes the partial Riemannian subgradient $\widetilde \nabla^{C_{ij}}_{\mathcal{M}_{X_{-ij}^k}} f(X^k)$ according to \eqref{eq-partialRiemsubgrad}. Lastly, it updates the block $X_{ij}$ via the retracted partial Riemannian subgradient step in~\eqref{eq-projsubmfdblk} and keeps the block $X_{-ij}$ unchanged.
We claim that line~\ref{line:update} of Algorithm~\ref{alg-RSSM} is well defined, in the sense that the two projections in~\eqref{eq-projsubmfdblk} are equal and yield a singleton. 
Indeed, let $\xi^k = \widetilde \nabla^{C_{ij}}_{\mathcal{M}_{X_{-ij}^k}} f(X^k) $ and $\gamma_k >0$. By definition of the polar decomposition-based retraction in \eqref{eq-polarretr} and the fact that $-\gamma_k \xi^k \in T_{X_{ij}^k}{\Stiefel n {p_{ij}}}$, the projection $\mathcal P_{\Stiefel n {p_{ij}}} \left( X_{ij}^k - \gamma_k \xi^k \right)$ is a singleton. Noting from \eqref{eq-TSsubmfdblk} that $(X_{-ij}^k)^\top (X_{ij}^k - \gamma_k \xi^k) = 0 $, we have $\mathcal P_{\Stiefel n {p_{ij}}} \left( X_{ij}^k - \gamma_k \xi^k \right) \in  \mathcal M_{X_{-ij}^k}$ and
\begin{align*}
    \min_{U \in \mathcal M_{X_{-ij}^k}}\| X_{ij}^k - \gamma_k \xi^k - U\|^2 = \min_{U \in \Stiefel n {p_{ij}}}\| X_{ij}^k - \gamma_k \xi^k - U\|^2.
\end{align*}
This, together with the fact $\mathcal M_{X_{-ij}^k} \subseteq \Stiefel n {p_{ij}}$, establishes the claim.

\begin{remark}
\label{remark: selrule}
Besides selecting the index pair $\{i,j\}$ randomly according to the uniform distribution, there are other randomized or deterministic selection rules, such as cyclic, Gauss-Southwell, etc. The uniformly random selection rule is theoretically advantageous due to its simplicity, especially in the Riemannian setting. Specifically, for deterministic selection rules, in order to analyze the aggregated effect of a full epoch of iterations that cover all the columns, one needs to relate the partial Riemannian subgradients at consecutive iterates, which typically requires the notion of parallel transport.
\end{remark}

\subsection{Per-Iteration Complexity}

The proposition below establishes the per-iteration complexity of our proposed RSSM in terms of number of floating point operations when we choose a \emph{uniform} partition of the column indices.

\begin{proposition}[Per-iteration Complexity] \label{prop-periteration}
Let $\mathfrak C = \{ \, C_1, \ldots, C_\ell \, \}$ be a partition of $[p]$ with $\ell = |\mathfrak{C}| \ge 2$ and $|C_i| \leq \lceil \frac{p}{\ell} \rceil$ for all $i \in [\ell]$. Given the partial Euclidean subgradient $\widetilde \nabla^{C_{ij}}\! f(X) \in \R^{n \times p_{ij}}$, an iteration of RSSM requires $\mathcal O\left(\frac{np^2}{\ell}\right)$ floating point operations.
\end{proposition}

\begin{proof}
Given a matrix $\xi_{ij} \in \R^{n \times p_{ij}}$, we can implement \eqref{eq-partialRiemsubgrad} as follows:
\begin{enumerate}
    \item Find $Y = X^\top \xi_{ij} \in \R^{p \times p_{ij}}$, which involves $\mathcal O(npp_{ij})$ floating point operations.
    \item Update $Y(C_{ij},:) \in \R^{p_{ij} \times p_{ij}}$ by $\mathrm{Sym}(Y(C_{ij},:))$, which involves $\mathcal O(p_{ij}^2)$ floating point operations.
    \item Perform the multiplication $XY \in \R^{n \times p_{ij}}$, which involves $\mathcal O(npp_{ij})$ floating point operations.
    \item Compute $\mathcal P_{T_{X_{ij}}\mathcal M_{X_{-ij}}}(\xi_{ij}) = \xi_{ij} - XY \in \R^{n\times p_{ij}}$, which involves $np_{ij}$ floating point operations.
\end{enumerate}
Hence, given the partial Euclidean subgradient $\widetilde \nabla^{C_{ij}} f(X) \in \R^{n \times p_{ij}}$, computing the partial Riemannian subgradient $\widetilde \nabla^{C_{ij}}_{\mathcal{M}_{X_{-ij}}} f(X) \in \R^{n \times p_{ij}}$ via \eqref{eq-partialRiemsubgrad} requires $\mathcal O(npp_{ij} + np_{ij} + p_{ij}^2) = \mathcal O(npp_{ij})$ (since $p_{ij} \leq p$) floating point operations.
Moreover, given a matrix $\Xi \in \R^{n \times p_{ij}}$, performing the polar decomposition $\mathcal P_{\Stiefel n {p_{ij}}}(\Xi)$ requires $\mathcal O(np^2_{ij})$ floating point operations; see, e.g., \cite[Chapter 8.6.4]{golub2013matrix}. This implies that the updates in \eqref{eq-projsubmfdblk} require $\mathcal O(np_{ij} + np_{ij}^2) = \mathcal O(np_{ij}^2)$ floating point operations. Putting the above together and noting that $p_{ij} \leq \lceil \frac{2p}{\ell} \rceil = \mathcal O(\frac{p}{\ell})$, we conclude that an iteration of RSSM requires $\mathcal O(npp_{ij} + np_{ij}^2) = \mathcal O(npp_{ij}) = \mathcal O(\frac{np^2}{\ell})$ floating point operations.
\end{proof}
{\begin{remark}\label{remark:per-iter-cost} In Proposition~\ref{prop-periteration}, we assume the availability of the partial Euclidean subgradient and evaluate the cost of each RSSM iteration in terms of the number of floating point operations (i.e., arithmetic complexity). To evaluate the total arithmetic complexity of each iteration, we need to take the cost of computing a partial Euclidean subgradient into account. Suppose that computing a full Euclidean subgradient of the objective function $f$ requires $\chi(f)$ floating point operations and that computing a partial Euclidean subgradient requires a fraction $\alpha(\ell)\in(0,1)$ of that cost. Then, the per-iteration arithmetic complexity of RSSM is
$\mathcal{O}( \chi(f) \cdot \alpha(\ell) + \tfrac{np^2}{\ell} )$. In particular, for the orthogonal dictionary learning and sparse principal component analysis applications considered in Section \ref{sect-exp}, it can be shown that $\chi(f)=\mathcal{O}(mn^2)$ and $\chi(f) = \mathcal{O}(\min\{m,n\}\cdot np)$, respectively, and $\alpha(\ell)= \tfrac{2}{\ell}$ for both applications due to the block-separable structures of their objective functions.
\end{remark}}

\section{Convergence Analysis} \label{sect-theory}

In this section, we study the convergence behavior of the proposed RSSM. We first prove a generalization of~\cite[Theorem 1]{li2021weakly}, which is instrumental to the analysis of subgradient-type methods for weakly convex optimization over a compact proximally smooth manifold. Recall that the Stiefel manifold and the submanifold block introduced in Definition~\ref{def-submfdblk} are both compact proximally smooth; see Section~\ref{subsect-proxsmooth} and Lemma~\ref{lem-projsubmfdblk}.

\begin{lemma}[Riemannian Subgradient Inequality] \label{lem-RSI}
Let $\mathcal M \subseteq \R^{n \times p}$ be a compact $R$-proximally smooth manifold. Suppose that $h : \mathbb R^{n \times p} \to \mathbb R$ is $\tau$-weakly convex for some $\tau \in \R$. Then, for any bounded convex neighborhood $\mathcal U$ that contains $\mathcal M$, there exists a constant $L > 0$ such that $h$ is $L$-Lipschitz continuous on $\mathcal U$. Moreover, for any $X, Y \in \mathcal M$, $\widetilde \nabla h(X) \in \pd h(X)$, and $\widetilde \nabla_{\mathcal M}h(X) = \mathcal P_{T_X \mathcal M}\widetilde \nabla h(X) \in \partial_{\mathcal M}h(X)$, we have
\begin{align} \label{eq-RSIv1}
h(Y) &\geq h(X) + \langle\, \widetilde \nabla_{\mathcal M} h(X), Y - X\,\rangle - \tfrac{\tau + \|\widetilde\nabla h(X)\|/R}{2}\,\|Y-X\|^2\\
&\geq h(X) + \langle\, \widetilde \nabla_{\mathcal M} h(X), Y - X\,\rangle - \tfrac{\tau + L/R}{2}\,\|Y-X\|^2. \label{eq-RSI}
\end{align}
\end{lemma}

We remark that Lemma~\ref{lem-RSI} applies even when $h$ is strongly convex, i.e., $\tau < 0$.

\begin{proof}[Proof of Lemma~\ref{lem-RSI}]
The Lipschitz continuity of $h$ on any bounded convex neighborhood that contains $\mathcal{M}$ follows directly from~\cite[Proposition 4.4]{vial1983strong}. Now, for any $X, Y \in \mathcal M$ and $\widetilde\nabla h(X) \in \pd h(X)$, we use~\eqref{eq-weakcvx} to get
\begin{align*}
h(Y) 
&\textstyle \geq h(X) + \langle\,\widetilde\nabla h(X), Y - X\,\rangle - \frac{\tau}{2}\|\,Y-X\,\|^2\\
&\textstyle = h(X) + \langle\,\mathcal P_{T_X \mathcal M}\widetilde\nabla h(X) + \mathcal P_{N_X \mathcal M}\widetilde\nabla h(X), Y - X\,\rangle - \frac{\tau}{2}\|\,Y-X\,\|^2.
\end{align*}
By \eqref{eq-proxsmooth}, we have
$$\langle\,\mathcal P_{N_X \mathcal M}\widetilde\nabla h(X), Y - X\,\rangle \geq -\frac{\|\mathcal P_{N_X \mathcal M}\widetilde \nabla h(X)\|}{2R}\,\|\,Y-X\,\|^2 \geq -\frac{\|\widetilde\nabla h(X)\|}{2R}\, \|\,Y-X\,\|^2.$$ 
Putting the above together and using the expression for $\pd_{\mathcal M} h(X)$ in~\eqref{eq-Rsubdiff}, we get~\eqref{eq-RSIv1}. Since the $L$-Lipschitz continuity of $h$ on $\mathcal{U}$ implies that $\| \widetilde\nabla h(X) \| \le L$, we get~\eqref{eq-RSI}.
\end{proof}

In what follows, we shall continue to work under the setting where the objective function $f$ of Problem~\eqref{mainpb} is $\tau$-weakly convex and $L$-Lipschitz continuous on some bounded convex neighborhood $\varOmega \subseteq \R^{n \times p}$ containing $\Stiefel n p$ for some $\tau \ge 0$ and $L>0$, and the set $\mathfrak{C} = \{ \, C_1,\ldots,C_{\ell} \,\}$ forms a partition of $[p]$ with $\ell = |\mathfrak{C}| \ge 2$.

\subsection{Averaging Operator} \label{subsect-coordrep}

Motivated by our uniformly random selection rule for the submanifold blocks, we now study the average of the \emph{lifted} partial Riemannian subgradients $\left\{\, \mathcal P_{T_{X_{ij}}\mathcal M_{X_{-ij}}}(\xi_{ij})I_{ij}^\top : \{i,j\} \in \binom{[\ell]}{2} \,\right\}$ on $T_X \Stiefel n p$. Note that given the partition $\mathfrak C$ of $[p]$, we have $X_{-ij}^\top (\mathcal P_{T_{X_{ij}}\mathcal M_{X_{-ij}}}(\xi_{ij}) ) = 0$ for all $\{i,j\} \in \binom{[\ell]}{2}$. Consider the linear operator $\mathcal A_X : \R^{n \times p} \to \R^{n \times p}$, called the \emph{averaging operator}, defined as
\begin{align} \label{eq-scalingoperator}
\mathcal A_X(\xi) = \frac{1}{\binom{\ell}{2}} \sum_{\{i,j\} \in \binom{[\ell]}{2}} (I-X_{-ij}X_{-ij}^\top)\xi_{ij} I_{ij}^\top.
\end{align}
Basically, the averaging operator $\mathcal A_X$ computes a simple average of the projections of partial Euclidean subgradients onto the orthogonal complement of the subspace spanned by the unaltered columns. Upon letting $X_{\perp} \in \Stiefel{n}{n-p}$ be such that $[X\ X_\perp] \in \Stiefel{n}{n}$, we can express $\mathcal A_X$ as
\begin{align}
\mathcal A_X(\xi) &\textstyle= X \left(\binom{\ell}{2}^{-1}\,\sum\limits_{\{i,j\} \in \binom{[\ell]}{2}} I_{ij}X_{ij}^\top \xi_{ij} I_{ij}^\top\right) + X_\perp \left(X_\perp^\top \left(\binom{\ell}{2}^{-1} \,\sum\limits_{\{i,j\} \in \binom{[\ell]}{2}} \xi_{ij} I_{ij}^\top\right)\right)\nonumber\\
&\textstyle= X \left(\binom{\ell}{2}^{-1}\, Q \boxdot (X^\top \xi)\right) + X_\perp \left(\frac{2}{\ell}\, X_\perp^\top \xi\right), \label{eq-scaling-coordrep}
\end{align}
where $Q = J + (\ell - 2) I$ and $\boxdot$ denotes the \emph{block Hadamard product} with respect to $\mathfrak C$, i.e., for $A = (a_{ij}) \in \R^{\ell \times \ell}$ and $B = (\,B_{ij}\,) \in \R^{p \times p}$ with $B_{ij} \in \R^{p_i \times p_j}$, we define $A \boxdot B = (\,a_{ij} B_{ij}\,) \in \R^{p \times p}$. In particular, we have
\begin{align*} 
\mathcal A_X (\xi) &= 
\textstyle
X  \left[\begin{array}{cccc}
\frac{2}{\ell}\,X_1^\top \xi_1 & \binom{\ell}{2}^{-1}X_1^\top \xi_2 & \cdots & \binom{\ell}{2}^{-1}X_1^\top \xi_\ell\\
\binom{\ell}{2}^{-1}X_2^\top \xi_1 & \frac{2}{\ell}\,X_2^\top \xi_2 & \cdots & \binom{\ell}{2}^{-1} X_2^\top \xi_\ell\\
\vdots & \vdots & \ddots & \vdots\\
\binom{\ell}{2}^{-1}X_\ell^\top \xi_1 & \binom{\ell}{2}^{-1}X_\ell^\top \xi_2 & \cdots & \frac{2}{\ell}\, X_\ell^\top \xi_\ell
\end{array}\right] + X_\perp \left(\frac{2}{\ell}\,X_\perp^\top \xi\right).
\end{align*}

We show in Lemma~\ref{lem-commutative} that $\mathcal A_X$ and $\mathcal P_{T_X \Stiefel n p}$ commute, and $\mathcal A_X \circ \mathcal P_{T_X\Stiefel n p}(\xi) = \binom{\ell}{2}^{-1} \sum_{\{i,j\} \in \binom{[\ell]}{2}}\mathcal P_{T_{X_{ij}}\mathcal M_{X_{-ij}}}(\xi_{ij})I_{ij}^\top$. Furthermore, we show in Lemma~\ref{lem-selfadjoint} that $\mathcal A_X$ is positive definite, which implies that the Mahalanobis inner product $\langle\,\xi, \eta\,\rangle_{\mathcal A_X^{-1}} = \langle\,\mathcal A_X^{-1}(\xi), \eta\,\rangle$ and the Mahalanobis norm $\|\xi\|_{\mathcal A_X^{-1}} = \sqrt{\langle\,\mathcal A_X^{-1}(\xi), \xi\,\rangle}$ are well defined.

The following lemma highlights the importance of $\langle\,\cdot\,,\,\cdot\,\rangle_{\mathcal A_X^{-1}}$.  

\begin{lemma}[Average of Partial Riemannian Subgradients] \label{lem-cindep}
Suppose that $\{i,j\} \sim \mathrm{Uniform}\binom{[\ell]}{2}$. For any $X \in \Stiefel{n}{p}$, $\widetilde\nabla f(X) \in \pd f(X)$, and $\eta \in \R^{n \times p}$, we have
\begin{align}
    \textstyle
    \E\left[ \left\langle\, \widetilde \nabla^{C_{ij}}_{\mathcal{M}_{X_{-ij}}} f(X)I_{ij}^\top, \eta\, \right\rangle_{\mathcal A_X^{-1}}\right] &= \langle\,\widetilde\nabla_{\Stiefel n p} f(X), \eta\,\rangle,\label{eq-cindepInnerProduct}\\ 
    \textstyle
    \E\left[ \left\| \widetilde \nabla^{C_{ij}}_{\mathcal{M}_{X_{-ij}}} f(X) I_{ij}^\top \right\|_{\mathcal A_X^{-1}}^2 \right] &= \|\widetilde\nabla_{\Stiefel n p} f(X)\|^2, \label{eq-cindepNorm2}
\end{align}
where $\widetilde\nabla_{\Stiefel n p} f(X) = \mathcal P_{T_{X}\Stiefel n p}(\widetilde\nabla f(X))$.
\end{lemma}
We remark that in~\eqref{eq-cindepInnerProduct} and~\eqref{eq-cindepNorm2}, both the full and partial Riemannian subgradients, $\widetilde\nabla_{\Stiefel n p} f(X) = \mathcal P_{T_{X}\Stiefel n p}(\widetilde\nabla f(X))$ and $\widetilde \nabla^{C_{ij}}_{\mathcal{M}_{X_{-ij}}} f(X) = \mathcal P_{T_{X_{ij}}\mathcal M_{X_{-ij}}}(\widetilde\nabla f(X)I_{ij})$ (the latter expression follows from~\eqref{eq:parRiesub} and Proposition~\ref{prop-structural}), depend on the choice of $\widetilde\nabla f(X)$.

\begin{proof}[Proof of Lemma~\ref{lem-cindep}]
Using Lemma \ref{lem-commutative} and the expression for $\widetilde \nabla^{C_{ij}}_{\mathcal{M}_{X_{-ij}}} f(X)$ mentioned above, we have
\begin{align*}
& \textstyle
\E\left[ \left\langle\, \widetilde \nabla^{C_{ij}}_{\mathcal{M}_{X_{-ij}}} f(X)I_{ij}^\top, \eta\, \right\rangle_{\mathcal A_X^{-1}}\,\right]
= \left\langle\frac{1}{\binom{\ell}{2}}\sum\limits_{\{i,j\} \in \binom{[\ell]}{2}} \mathcal P_{T_{X_{ij}}\mathcal M_{X_{-ij}}}(\widetilde\nabla f(X)I_{ij})I_{ij}^\top, \eta\,\right\rangle_{\mathcal A_X^{-1}}\\
&\textstyle
= \left\langle \mathcal A_X \circ \mathcal P_{T_X \Stiefel n p}(\widetilde\nabla f(X)), \eta \right\rangle_{\mathcal A_X^{-1}} = \langle\,\widetilde\nabla_{\Stiefel n p} f(X), \eta\,\rangle,
\end{align*}
which shows that \eqref{eq-cindepInnerProduct} holds. Next, for $\{i,j\} \in \binom{[\ell]}{2}$, we define $\mathcal{Q}_{ij}:\R^{n \times p} \rightarrow \R^{n \times p}$ as $\mathcal{Q}_{ij}(\xi) = \mathcal P_{T_{X_{ij}}\mathcal M_{X_{-ij}}}(\xi I_{ij})I_{ij}^\top$. Since $\mathcal P_{T_{X_{ij}}\mathcal M_{X_{-ij}}} \circ \mathcal P_{T_{X_{ij}}\mathcal M_{X_{-ij}}} = \mathcal P_{T_{X_{ij}}\mathcal M_{X_{-ij}}}$, we have $\mathcal{Q}_{ij} \circ \mathcal{Q}_{ij} = \mathcal{Q}_{ij}$. Moreover, by \eqref{eq-commutative1}, we have $\mathcal{A}_X \circ \mathcal{Q}_{ij} = \mathcal{Q}_{ij} \circ \mathcal{A}_X$. This, together with Lemma~\ref{lem-selfadjoint}, implies that $\mathcal{A}_X^{-1} \circ \mathcal{Q}_{ij} = \mathcal{Q}_{ij} \circ \mathcal{A}_X^{-1}$. It follows that
\begin{align*}
&\textstyle
\left\langle\, \widetilde \nabla^{C_{ij}}_{\mathcal{M}_{X_{-ij}}} f(X)I_{ij}^\top, \widetilde \nabla f(X)\, \right\rangle_{\mathcal A_X^{-1}} = \left\langle\,\mathcal A_X^{-1} \left(\mathcal P_{T_{X_{ij}}\mathcal M_{X_{-ij}}}(\widetilde\nabla f(X)I_{ij})I_{ij}^\top\right), \widetilde\nabla f(X) \right\rangle \\
&\textstyle= \left\langle\,\mathcal A_X^{-1} \left( \mathcal{Q}_{ij} \circ \mathcal{Q}_{ij}(\widetilde\nabla f(X)) \right), \widetilde\nabla f(X) \right\rangle = \left\langle\,\mathcal Q_{ij} \left( \mathcal{A}_X^{-1} \circ \mathcal{Q}_{ij}(\widetilde\nabla f(X)) \right), \widetilde\nabla f(X) \right\rangle \\
&\textstyle= \left\langle\,\mathcal P_{T_{X_{ij}}\mathcal M_{X_{-ij}}} \left(\mathcal A_X^{-1} \left(\mathcal P_{T_{X_{ij}}\mathcal M_{X_{-ij}}}(\widetilde\nabla f(X)I_{ij})I_{ij}^\top\right)I_{ij}\right) I_{ij}^\top, \widetilde\nabla f(X) \right\rangle\\
&\textstyle= \left\langle\,\mathcal A_X^{-1} \left(\mathcal P_{T_{X_{ij}}\mathcal M_{X_{-ij}}}(\widetilde\nabla f(X)I_{ij})I_{ij}^\top\right), \mathcal P_{T_{X_{ij}}\mathcal M_{X_{-ij}}}(\widetilde\nabla f(X)I_{ij})I_{ij}^\top \right\rangle\\
&\textstyle= \left\| \widetilde \nabla^{C_{ij}}_{\mathcal{M}_{X_{-ij}}} f(X) I_{ij}^\top \right\|_{\mathcal A_X^{-1}}^2.
\end{align*}
By taking $\eta = \widetilde\nabla f(X)$ in~\eqref{eq-cindepInnerProduct}, we obtain~\eqref{eq-cindepNorm2}.
\end{proof}

Our next lemma provides an estimate of the change in the local metrics at successive iterates generated by RSSM.

\begin{lemma}[Metric Comparison] \label{lem-almostNonexpansive}
Suppose that $\{i,j\} \sim \mathrm{Uniform}\binom{[\ell]}{2}$. Given $X \in \Stiefel n p$, $\widetilde\nabla f(X) \in \pd f(X)$, and {$\gamma>0$}, let $X^{+} \in \Stiefel n p$ be the (unique) matrix satisfying $X^{+}_{ij} = \mathcal P_{\mathcal M_{X_{-ij}}} \left( X_{ij} - \gamma  \widetilde \nabla^{C_{ij}}_{\mathcal{M}_{X_{-ij}}} f(X) \right)$ and $X_{-ij}^{+} = X_{-ij}$. 
Then, for any $Y \in \Stiefel n p$, we have
\begin{align} 
    \textstyle \E&\textstyle \left[\|Y-X^{+}\|_{\mathcal A_{X^{+}}^{-1}}^2\,\right] 
    \leq \E\left[\|Y-X^{+}\|_{\mathcal A_{X}^{-1}}^2\,\right] + 2(\ell-2)\gamma L \|Y-X\|^2 \nonumber\\
    &\textstyle\quad + (\ell-2) \gamma^2 L^2\left(\|Y-X\|^2 + 2\|Y-X\|\right) + (\ell-2) \gamma^3 L^3(\|Y-X\|^2+1). \label{eq-almostNonexpansive}
\end{align}
\end{lemma}

Before proving Lemma~\ref{lem-almostNonexpansive}, we present a lemma that characterizes the operator norm of the difference between two orthogonal projections, whose proof can be found in~\cite[Theorem 2.5.1]{golub2013matrix}. 

\begin{lemma} \label{lem-diffinprojections}
Let $X, Y \in \Stiefel n p$ and $X_\perp, Y_\perp \in \Stiefel n {n-p}$ be such that $[X\ X_\perp]$, $[Y\ Y_\perp] \in \Stiefel{n}{n}$. Then, we have $\|XX^\top-YY^\top\|_{\mathsf{op}} = \|X_\perp^\top Y\|_{\mathsf{op}} = \|Y_\perp^\top X\|_{\mathsf{op}}$.
\end{lemma}

\begin{proof}[Proof of Lemma \ref{lem-almostNonexpansive}]
Consider the linear operator $\Psi:\R^{n \times p} \rightarrow \R^{n \times p}$ defined as $\Psi_X(\xi) = X\left[\left(J-I\right) \boxdot (X^\top \xi)\right]$. It is easily seen that $\Psi_X$ is self-adjoint. By \eqref{eq-diffscaling1}, for any $Y \in \Stiefel{n}{p}$, we have
\begin{align*}
&\|Y-X^{+}\|_{\mathcal A_{X^{+}}^{-1}}^2 - \|Y-X^{+}\|_{\mathcal A_{X}^{-1}}^2\\
&= \langle\,\mathcal A_{X^{+}}^{-1}(Y-X^{+}), Y-X^{+}\rangle - \langle\,\mathcal A_{X}^{-1}(Y-X^{+}), Y-X^{+}\rangle\\ 
&\textstyle= \frac{\ell(\ell-2)}{2}\,\langle\,(\Psi_{X^{+}} - \Psi_X)(Y-X^{+}), Y-X^{+}\,\rangle\\
&\textstyle= \frac{\ell(\ell-2)}{2}\,\left(\langle\,(\Psi_{X^{+}} - \Psi_X)(Y-X), Y-X\,\rangle\right.\\ 
&\qquad \left.\ \ +\, 2\,\langle\,(\Psi_{X^{+}} - \Psi_X)(X-X^{+}), Y-X\,\rangle + \langle\,(\Psi_{X^{+}} - \Psi_X)(X-X^{+}), X-X^{+}\,\rangle\right).
\end{align*}
Now, for any $\xi \in \R^{n\times p}$, we have
\begin{align} \label{eq-diffPsi}
(\Psi_{X^{+}}-\Psi_X)(\xi) =(X_i^{+}X_i^{+}{}^\top - X_i X_i^\top)\xi_{-i}I_{-i}^\top + (X_j^{+}X_j^{+}{}^\top - X_j X_j^\top)\xi_{-j}I_{-j}^\top
\end{align}
and, using~\eqref{eq-lipslike},
\begin{align*}
&\langle\,(\Psi_{X^{+}}-\Psi_X)(\xi), \xi\,\rangle = \left(\|X_i^{+}{}^\top \xi_{-i}\|^2 - \|X_{i}^\top \xi_{-i}\|^2\right) + \left(\|X_j^{+}{}^\top \xi_{-j}\|^2 - \|X_{j}^\top \xi_{-j}\|^2\right)\\
&= \langle\,(X_i^{+}+X_i)^\top \xi_{-i}, (X_i^{+}-X_i)^\top \xi_{-i}\,\rangle + \langle\,(X_j^{+}+X_j)^\top \xi_{-j}, (X_j^{+}-X_j)^\top \xi_{-j}\,\rangle\\
&\leq \|(X_i^{+}+X_i)^\top \xi_{-i}\| \cdot \|X_i^{+}-X_i\| \cdot \|\xi_{-i}\| + \|(X_j^{+}+X_j)^\top \xi_{-j}\| \cdot \|X_j^{+}-X_j\| \cdot \|\xi_{-j}\|\\
&\leq \|(X_i^{+}+X_i)^\top \xi\| \cdot \|X_i^{+}-X_i\| \cdot \|\xi\| + \|(X_j^{+}+X_j)^\top \xi\| \cdot \|X_j^{+}-X_j\| \cdot \|\xi\|\\
&\leq \|(X_{ij}^{+}+X_{ij})^\top \xi\| \cdot \|X_{ij}^{+}-X_{ij}\| \cdot \|\xi\| \\
&\leq \gamma \left\|\widetilde \nabla^{C_{ij}}_{\mathcal{M}_{X_{-ij}}} f(X)\right\| \cdot \|(X_{ij}^{+}+X_{ij})^\top \xi\| \cdot \|\xi\|.
\end{align*}
{The second-order boundedness condition \eqref{eq-secondbdd} implies} that
\begin{align*}
\|(X_{ij}^{+}+X_{ij})^\top \xi\| &= \left\| \left( 2X_{ij} - \gamma \widetilde \nabla^{C_{ij}}_{\mathcal{M}_{X_{-ij}}} f(X) + X_{ij}^{+} - X_{ij} + \gamma \widetilde \nabla^{C_{ij}}_{\mathcal{M}_{X_{-ij}}} f(X) \right)^\top \xi \right\| \\ 
&\leq 2\| X_{ij}^\top\xi \| + \gamma \left\|\widetilde \nabla^{C_{ij}}_{\mathcal{M}_{X_{-ij}}} f(X)\right\| \cdot \|\xi\| + \gamma^2 \left\|\widetilde \nabla^{C_{ij}}_{\mathcal{M}_{X_{-ij}}} f(X)\right\|^2 \cdot \|\xi\|.
\end{align*}
It follows that
\begin{align}
\langle\,(\Psi_{X^{+}}-\Psi_X)(\xi), \xi\,\rangle &\leq 2\gamma \left\|\widetilde \nabla^{C_{ij}}_{\mathcal{M}_{X_{-ij}}} f(X)\right\| \cdot \|X_{ij}^\top\xi\| \cdot \|\xi\| + \gamma^2 \left\|\widetilde \nabla^{C_{ij}}_{\mathcal{M}_{X_{-ij}}} f(X)\right\|^2 \cdot \|\xi\|^2 \nonumber \\
&\qquad+ \gamma^3 \left\|\widetilde \nabla^{C_{ij}}_{\mathcal{M}_{X_{-ij}}} f(X)\right\|^3 \cdot \|\xi\|^2. \label{eq-diffPsi1}
\end{align}
By applying \eqref{eq-diffPsi} with $\xi = X-X^{+}$ and using Lemma \ref{lem-diffinprojections}, we get
\begin{align}
&\|(\Psi_{X^{+}} - \Psi_X)(X-X^{+})\| \leq {\max_{C \in \{ij, i, j\}} \|X_{C}^{+}X_{C}^{+}{}^\top - X_{C}X_{C}^\top\|_{\mathsf{op}}} \cdot \|X-X^{+}\| \nonumber \\
&\leq {\max_{C \in \{ij, i, j\}} \|(I-X_{C}X_{C}^\top)X_{C}^{+}\|_{\mathsf{op}}} \cdot  \|X_{ij}-X_{ij}^{+}\|\nonumber\\
&= {\max_{C \in \{ij, i, j\}} \|(I-X_{C}X_{C}^\top)(X_{C}^{+}-X_{C})\|_{\mathsf{op}}} \cdot \|X_{ij}-X_{ij}^{+}\|\nonumber\\ 
&\leq \|X_{ij}-X_{ij}^{+}\|^2 \leq \gamma^2 \left\|\widetilde \nabla^{C_{ij}}_{\mathcal{M}_{X_{-ij}}} f(X)\right\|^2. \label{eq-diffPsi2}
\end{align}
By taking $\xi = Y-X$ in \eqref{eq-diffPsi1} and combining the result with \eqref{eq-diffPsi2}, we obtain
{\allowdisplaybreaks
\begin{align}
        &\|Y-X^{+}\|_{\mathcal A_{X^{+}}^{-1}}^2 - \|Y-X^{+}\|_{\mathcal A_{X}^{-1}}^2 \label{ineq:recursion_1}\\
        &\textstyle\leq \frac{\ell(\ell-2)}{2}\,\left(\langle\,(\Psi_{X^{+}} - \Psi_X)(Y-X), Y-X\,\rangle\right.\nonumber\\ 
        &\  \left. +\, 2\,\|(\Psi_{X^{+}} - \Psi_X)(X-X^{+})\| \cdot \|Y-X\| + \| (\Psi_{X^{+}} - \Psi_X)(X-X^{+})\| \cdot \| X-X^{+}\|\right)\nonumber\\
        &\textstyle\leq \frac{\ell(\ell-2)}{2}\,\left(2\gamma \left\| \widetilde \nabla^{C_{ij}}_{\mathcal{M}_{X_{-ij}}} f(X) \right\| \cdot \|X_{ij}^\top (Y-X)\| \cdot \|Y-X\| \right. \nonumber\\
        &\textstyle\qquad\qquad\quad + \gamma^2 \left\|\widetilde \nabla^{C_{ij}}_{\mathcal{M}_{X_{-ij}}} f(X) \right\|^2 \cdot \|Y-X\|^2 + \gamma^3 \left\|\widetilde \nabla^{C_{ij}}_{\mathcal{M}_{X_{-ij}}} f(X)\right\|^3 \cdot \|Y-X\|^2 \nonumber\\ 
        &\textstyle\qquad\qquad\quad \left.+\, 2\gamma^2 \left\| \widetilde \nabla^{C_{ij}}_{\mathcal{M}_{X_{-ij}}} f(X) \right\|^2 \cdot \|Y-X\| + \gamma^3 \left\| \widetilde \nabla^{C_{ij}}_{\mathcal{M}_{X_{-ij}}} f(X) \right\|^3\right)\nonumber\\
        &\textstyle\leq (\ell-2)\left(2\gamma \left\| \widetilde \nabla^{C_{ij}}_{\mathcal{M}_{X_{-ij}}} f(X)I_{ij}^\top \right\|_{\mathcal A_{X}^{-1}} \cdot \sqrt{\frac{\ell}{2}}\|X_{ij}^\top (Y-X)\| \cdot \|Y-X\| \right.\nonumber\\
        &\textstyle\qquad\qquad\quad\, + \gamma^2 (1+\gamma L) \left\|\widetilde \nabla^{C_{ij}}_{\mathcal{M}_{X_{-ij}}} f(X)I_{ij}^\top \right\|_{\mathcal A_{X}^{-1}}^2 \cdot \|Y-X\|^2 \nonumber \\
        &\textstyle\qquad\qquad\quad\, \left. + 2\gamma^2 \left\| \widetilde \nabla^{C_{ij}}_{\mathcal{M}_{X_{-ij}}} f(X)I_{ij}^\top \right\|_{\mathcal A_{X}^{-1}}^2 \cdot \|Y-X\| + \gamma^3L \left\| \widetilde \nabla^{C_{ij}}_{\mathcal{M}_{X_{-ij}}} f(X)I_{ij}^\top \right\|_{\mathcal A_{X}^{-1}}^2 \right),\nonumber 
\end{align}}
where the last inequality follows from Lemma~\ref{lem-selfadjoint} and the fact that $\left\|\widetilde \nabla^{C_{ij}}_{\mathcal{M}_{X_{-ij}}} f(X)\right\| \le L$. Now, we take expectation on both sides of the above inequality with respect to $\{i,j\} \sim \mathrm{Uniform}\binom{[\ell]}{2}$. Observe that
\begin{align*}
\textstyle \E\left[\|X_{ij}^\top (Y-X)\|^2\right] 
&\textstyle= \frac{1}{\binom{\ell}{2}}  \sum\limits_{\{i,j\} \in \binom{[\ell]}{2}} 
\left(\|X_{i}^\top (Y-X)\|^2 + \|X_{j}^\top (Y-X)\|^2\right) \\
&\textstyle= \frac{2}{\ell} \|X^\top (Y-X)\|^2.
\end{align*}
Moreover, by the Cauchy--Schwarz inequality, we get
\begin{align*}
&\textstyle\E\left[\left\| \widetilde \nabla^{C_{ij}}_{\mathcal{M}_{X_{-ij}}} f(X)I_{ij}^\top \right\|_{\mathcal A_{X}^{-1}} \cdot \sqrt{\frac{\ell}{2}}\|X_{ij}^\top (Y-X)\| \right]\\ 
&\textstyle\leq \sqrt{\E\left[\left\| \widetilde \nabla^{C_{ij}}_{\mathcal{M}_{X_{-ij}}} f(X)I_{ij}^\top \right\|_{\mathcal A_{X}^{-1}}^2\right]} \cdot \sqrt{\frac{\ell}{2} \E\left[\|X_{ij}^\top (Y-X)\|^2\right]}. 
\end{align*}
This, together with~\eqref{eq-cindepNorm2},~\eqref{ineq:recursion_1}, and the fact that $\|\widetilde\nabla_{\Stiefel n p} f(X)\| \le L$, yields
\begin{align*}
&\textstyle \E\left[\|Y-X^{+}\|_{\mathcal A_{X^{+}}^{-1}}^2\,\right]\\
&\textstyle\leq \E\left[\|Y-X^{+}\|_{\mathcal A_{X}^{-1}}^2\,\right] + 2(\ell-2) \gamma  \|\widetilde\nabla_{\Stiefel n p}f(X)\| \cdot \|X^\top (Y-X)\| \cdot \|Y-X\| \\
&\textstyle\quad + 
(\ell-2)\gamma^2(1+\gamma L)\|\widetilde\nabla_{\Stiefel n p}f(X)\|^2\cdot\|Y-X\|^2 \\
&\textstyle \quad + 2(\ell-2)\gamma^2 \|\widetilde\nabla_{\Stiefel n p}f(X)\|^2\cdot\|Y-X\| + (\ell-2) \gamma^3L \|\widetilde\nabla_{\Stiefel n p}f(X)\|^2 \\
&\textstyle\leq \E\left[\|Y-X^{+}\|_{\mathcal A_{X}^{-1}}^2\,\right] + 2(\ell-2)\gamma L \|Y-X\|^2 \\
&\textstyle\quad + (\ell-2) \gamma^2 L^2\left(\|Y-X\|^2 + 2\|Y-X\|\right) + (\ell-2) \gamma^3 L^3(\|Y-X\|^2+1),
\end{align*}
as desired.
\end{proof}

\subsection{Adaptive Moreau Envelope, Adaptive Proximal Map, and Surrogate Stationarity Measure}

For nonsmooth optimization over a compact embedded submanifold, the work \cite{li2021weakly} defines Riemannian analogues of the Moreau envelope and proximal map to facilitate the convergence analysis of Riemannian subgradient-type methods. 
However, these constructs are not well suited for capturing the effects of the random submanifold block updates in our proposed RSSM. To address this difficulty, we introduce adaptive versions of the Moreau envelope and proximal map in~\cite{li2021weakly}, in which we use the norm induced by the inverse of the averaging operator to measure proximity.
Specifically, for Problem~\eqref{mainpb}, given the partition $\mathfrak C$ of $[p]$, we define the \emph{adaptive Moreau envelope} and \emph{adaptive proximal map} at $X\in \Stiefel n p$ as
\begin{equation*}
    f_\lambda^{\mathfrak C}(X) = \min_{Y \in \Stiefel n p}\,\left\{\,f(Y) + \textstyle{\frac{1}{2\lambda}}\,\|Y-X\|_{\mathcal A_X^{-1}}^2\,\right\}
\end{equation*}
and
\begin{equation}
    P_{\lambda f}^{\mathfrak C}(X) = \argmin_{Y \in \Stiefel n p}\,\left\{\,f(Y) + \textstyle{\frac{1}{2\lambda}}\,\|Y-X\|_{\mathcal A_X^{-1}}^2\,\right\},\label{eq-bregmanproximal}
\end{equation}
respectively. As we shall see, the above constructs allow us to use Lemma~\ref{lem-cindep} to study the average effects of the random submanifold block updates in RSSM. We first show that for small $\lambda>0$, the adaptive proximal map $P_{\lambda f}^{\mathfrak C}$ is single-valued and satisfies a Lipschitz-like property; cf.~\cite[Lemma 4.2]{wang2023decentralized}.

\begin{lemma} \label{lem-proxlips}
Consider the setting of Lemma~\ref{lem-almostNonexpansive}. If $\lambda \in (0, \frac{\ell}{2(\tau + {(2\ell-1)L})})$, then the adaptive proximal map $P_{\lambda f}^{\mathfrak C}$ is single-valued and satisfies 
\begin{equation}
    \| P_{\lambda f}^{\mathfrak C} (X) - P_{\lambda f}^{\mathfrak C}(X^+)\| \le \tfrac{\sqrt{2}\binom{\ell}{2}}{\frac{\ell}{2}-\lambda(\tau + (2\ell-1)L)} \| X - X^+ \|. \label{eq-lipslike2}
\end{equation}
\end{lemma}

\begin{proof}
Let $h_X:\R^{n \times p} \rightarrow \R$ be the function defined as $h_X(Z) = f(Z) + \frac{1}{2\lambda}\|Z-X\|_{\mathcal A_X^{-1}}^2$. By Lemma~\ref{lem-selfadjoint}, when $\lambda < \tfrac{\ell}{2\tau}$, the function $h_X$ is $\left(\frac{\ell}{2\lambda} - \tau\right)$-strongly convex. Moreover, for any $Z \in \R^{n \times p}$, we have $\pd h_X(Z) = \pd f(Z) + \frac{1}{\lambda}\,\mathcal A_X^{-1}(Z-X)$.

If $Z \in P_{\lambda f}^{\mathfrak C}(X) = \argmin_{Z' \in \Stiefel n p} h_X(Z')$, then $h_X(Z) = f(Z) + \frac{1}{2\lambda} \|Z-X\|_{\mathcal A_X^{-1}}^2 \leq f(X)$. This, together with Lemma~\ref{lem-selfadjoint}, implies that $\frac{\ell}{2} \cdot \frac{1}{2\lambda} \|Z-X\|^2 \leq \frac{1}{2\lambda} \|Z-X\|_{\mathcal A_X^{-1}}^2 \leq f(X) - f(Z) \leq L\,\|Z-X\|$, or equivalently, $\|Z-X\| \leq \frac{4\lambda L}{\ell}$. Hence, we have
\begin{align}
\|\widetilde\nabla h_X(Z)\| 
&\textstyle\leq \|\widetilde\nabla f(Z)\| + \frac{1}{\lambda} \|\mathcal A_X^{-1}\|_{\mathsf{op}} \cdot \|Z-X\| \leq L + \frac{1}{\lambda} \binom{\ell}{2} \frac{4\lambda L}{\ell}= (2\ell-1) L\label{eq:hX-subgrad-bd}
\end{align}
for all $\widetilde\nabla f(Z) \in \pd f(Z)$ and $\widetilde\nabla h_X(Z) = \widetilde\nabla f(Z) + \frac{1}{\lambda}\,\mathcal A_X^{-1}(Z-X) \in \pd h_X(Z)$.

Now, let $Y,Z \in P_{\lambda f}^{\mathfrak C}(X)$ be arbitrary. The first-order optimality condition of~\eqref{eq-bregmanproximal} implies that $0 \in \pd_{\Stiefel{n}{p}} h_X(Y)$ and $0 \in \pd_{\Stiefel{n}{p}} h_X(Z)$. This, together with the Riemannian subgradient inequality~\eqref{eq-RSIv1}, the subgradient bound~\eqref{eq:hX-subgrad-bd}, and the 1-proximal smoothness of $\Stiefel{n}{p}$, yields
\begin{align*}
h_X(Y) &\textstyle\geq h_X(Z) + \frac{\frac{\ell}{2\lambda}-\tau - (2\ell-1)L}{2}\|Y-Z\|^2, \\ 
h_X(Z) &\textstyle\geq h_X(Y) + \frac{\frac{\ell}{2\lambda}-\tau - (2\ell-1)L}{2}\|Y-Z\|^2.
\end{align*}
It follows that when $\lambda < \frac{\ell}{2(\tau + (2\ell-1)L)}$, we have $\|Y-Z\|^2 \leq 0$, i.e., $Y=Z$. This shows that $P_{\lambda f}^{\mathfrak C}$ is single-valued.

Using the Riemannian subgradient inequality~\eqref{eq-RSIv1} again, we have
\begin{align*}
h_X(P_{\lambda f}^{\mathfrak C}(X^{+})) &\textstyle
\geq h_X(P_{\lambda f}^{\mathfrak C}(X)) + \frac{\frac{\ell}{2\lambda}-\tau-(2\ell-1)L}{2} \|P_{\lambda f}^{\mathfrak C}(X^{+})-P_{\lambda f}^{\mathfrak C}(X)\|^2,\\
h_{X^{+}}(P_{\lambda f}^{\mathfrak C}(X)) &\textstyle
\geq h_{X^{+}}(P_{\lambda f}^{\mathfrak C}(X^{+})) + \frac{\frac{\ell}{2\lambda}-\tau-(2\ell-1)L}{2} \|P_{\lambda f}^{\mathfrak C}(X)-P_{\lambda f}^{\mathfrak C}(X^{+})\|^2.
\end{align*}
Summing the above two inequalities gives
\begin{align*}
&\textstyle\left(\frac{\ell}{2\lambda}-\tau-(2\ell-1)L\right)\|P_{\lambda f}^{\mathfrak C}(X)-P_{\lambda f}^{\mathfrak C}(X^{+})\|^2\\
&\leq h_X(P_{\lambda f}^{\mathfrak C}(X^{+})) - h_X(P_{\lambda f}^{\mathfrak C}(X)) + h_{X^{+}}(P_{\lambda f}^{\mathfrak C}(X)) - h_{X^{+}}(P_{\lambda f}^{\mathfrak C}(X^{+}))\\
&\textstyle
= \frac{1}{2\lambda} \bigg(\,\|P_{\lambda f}^{\mathfrak C}(X^{+})-X\|_{\mathcal A_X^{-1}}^2 - \|P_{\lambda f}^{\mathfrak C}(X)-X\|_{\mathcal A_X^{-1}}^2 \\
&\textstyle\quad\quad +\, \|P_{\lambda f}^{\mathfrak C}(X)-X^{+}\|_{\mathcal A_{X^{+}}^{-1}}^2 - \|P_{\lambda f}^{\mathfrak C}(X^{+})-X^{+}\|_{\mathcal A_{X^{+}}^{-1}}^2\,\bigg)\\
&\textstyle
= \frac{1}{2\lambda} \left(\,\|P_{\lambda f}^{\mathfrak C}(X^{+})\|_{\mathcal A_X^{-1}}^2 - \|P_{\lambda f}^{\mathfrak C}(X^{+})\|_{\mathcal A_{X^{+}}^{-1}}^2 + \|P_{\lambda f}^{\mathfrak C}(X)\|_{\mathcal A_{X^{+}}^{-1}}^2 - \|P_{\lambda f}^{\mathfrak C}(X)\|_{\mathcal A_X^{-1}}^2\right.\\
&\textstyle
\qquad
\left.-\, \ell\, \langle\,P_{\lambda f}^{\mathfrak C}(X^{+}), X\,\rangle + \ell\, \langle\,P_{\lambda f}^{\mathfrak C}(X), X\,\rangle -\, \ell\, \langle\,P_{\lambda f}^{\mathfrak C}(X), X^{+}\,\rangle + \ell\, \langle\,P_{\lambda f}^{\mathfrak C}(X^{+}), X^{+}\,\rangle\,\right)\\
&\textstyle
\leq \frac{\sqrt{2}\ell(\ell-2)}{2\lambda} \|X^{+}-X\| \cdot \| P_{\lambda f}^{\mathfrak C}(X^{+}) - P_{\lambda f}^{\mathfrak C}(X)\| + \frac{\ell}{2\lambda}\,\langle\,P_{\lambda f}^{\mathfrak C}(X^{+}) - P_{\lambda f}^{\mathfrak C}(X), X^{+} - X\,\rangle\\
&\textstyle
\leq \frac{{\sqrt{2}}}{\lambda}\binom{\ell}{2} \|X^{+}-X\| \cdot \|P_{\lambda f}^{\mathfrak C}(X^{+}) - P_{\lambda f}^{\mathfrak C}(X)\|,
\end{align*}
where the second equality follows from the fact that $\mathcal A_X^{-1}(X) = \frac{\ell}{2}\,X$ and $\mathcal A_{X^{+}}^{-1}(X^{+}) = \frac{\ell}{2}\,X^{+}$, and the second inequality follows from
{\eqref{eq-almostNonexpansive2} (taking $\xi=P_{\lambda f}^{\mathfrak C}(X^{+})$ and $\zeta=P_{\lambda f}^{\mathfrak C}(X)$ therein)}
and the Cauchy--Schwarz inequality. This establishes the Lipschitz-like property \eqref{eq-lipslike2}.
\end{proof}

Now, motivated by the definition of the adaptive proximal map $P_{\lambda f}^{\mathfrak C}$ in~\eqref{eq-bregmanproximal}, let us consider the map $\Theta_{\lambda f}^{\mathfrak C}:\R^{n \times p} \rightarrow \R_+$ defined as $\Theta_{\lambda f}^{\mathfrak C}(X) = \textstyle{\frac{1}{
\lambda}}\|P_{\lambda f}^{\mathfrak C}(X)-X\|_{\mathcal A_X^{-1}}$. The following proposition shows that $\Theta_{\lambda f}^{\mathfrak C}$ can be viewed as a \emph{surrogate stationarity measure} of Problem~\eqref{mainpb}.


\begin{proposition}[Surrogate Stationarity Measure] \label{prop-stationarymeasure}
    Let $\lambda \in (0, \frac{\ell}{2(\tau + {(2\ell-1)L})})$ be fixed, so that $P_{\lambda f}^{\mathfrak C}$ is single-valued. Then, for any $X \in \Stiefel{n}{p}$, the following assertions hold: 
     
    \begin{enumerate}[label=(\alph*)]   
        \item\label{prop-stationarymeasure:a} 
        $\dist(0,\pd_{\Stiefel n p}f(P_{\lambda f}^{\mathfrak C}(X))) \leq \sqrt{\binom{\ell}{2}}\frac{1}{\lambda} \|P_{\lambda f}^{\mathfrak C}(X)-X\|_{\mathcal A_X^{-1}} = \sqrt{\binom{\ell}{2}}\, \Theta_{\lambda f}^{\mathfrak C}(X)$ and  \\
        $\frac{\frac{\ell}{2}-\lambda(\tau+L)}{\sqrt{2\ell}}\Theta_{\lambda f}^{\mathfrak C}(X) = \frac{\frac{\ell}{2}-\lambda(\tau+L)}{\sqrt{2\ell}\lambda}\|P_{\lambda f}^{\mathfrak C}(X)-X\|_{\mathcal A_X^{-1}} \leq \dist(0,\pd_{\Stiefel n p}f(X))$.
         
        \item\label{prop-stationarymeasure:b} $\Theta_{\lambda f}^{\mathfrak C}(X) = 0\ \Longleftrightarrow\ X = P_{\lambda f}^{\mathfrak C}(X)\ \Longleftrightarrow \ 0 \in \pd_{\Stiefel n p}f(X)$.
    \end{enumerate}    
\end{proposition}


\begin{proof}
    We first prove assertion~\ref{prop-stationarymeasure:a}. Let $Z = P_{\lambda f}^{\mathfrak C}(X)$. By the first-order optimality condition of~\eqref{eq-bregmanproximal}, we have $\mathcal P_{T_Z \Stiefel n p}\!\left(\frac{1}{\lambda}\,\mathcal A_{X}^{-1}(X - Z)\right) \in \pd_{\Stiefel n p}f(Z)$. This, together with Lemma~\ref{lem-selfadjoint}, implies that $\dist(0,\pd_{\Stiefel n p}f(Z)) \leq \frac{1}{\lambda} \,\|\mathcal A_X^{-1}(X-Z)\| \leq \frac{1}{\lambda}\sqrt{\binom{\ell}{2}}\|X-Z\|_{\mathcal A_X^{-1}}$. Now, let $V \in \pd_{\Stiefel n p}f(X)$ be arbitrary. By the Riemannian subgradient inequality~\eqref{eq-RSI} and Lemma~\ref{lem-selfadjoint}, we have
        \begin{align*}
            \langle\,V,X-Z\,\rangle &\textstyle\geq f(X) - f(Z) - \frac{\tau+L}{2}\|Z-X\|^2 \\
            &\textstyle\geq \frac{1}{2\lambda}\|Z-X\|_{\mathcal A_X^{-1}}^2 - \frac{\tau+L}{2}\|Z-X\|^2\\ 
            &\textstyle\geq \left(\frac{1}{2\lambda}-\frac{\tau+L}{2} \cdot \frac{2}{\ell}\right)\|Z-X\|_{\mathcal A_X^{-1}} \cdot \sqrt{\frac{\ell}{2}}\|Z-X\|.
        \end{align*}
        This implies that $\|V\| \geq \sqrt{\frac{\ell}{2}} \cdot \frac{1}{\ell}\left(\frac{\ell}{2\lambda}-\tau-L\right)\|Z-X\|_{\mathcal A_X^{-1}}$, where $\frac{\ell}{2\lambda}-\tau-L > (2\ell-1)L - L = 2(\ell-1) L > 0$.

        Assertion~\ref{prop-stationarymeasure:b} follows directly from assertion~\ref{prop-stationarymeasure:a}.
\end{proof}

\subsection{Convergence Rate of RSSM} \label{subsect-convstationary}
Our final goal is to establish the convergence rate of RSSM. We first prove a recursion lemma.

\begin{lemma}[Recursion Lemma]
\label{lem-basicrecursion}
Let $\{X^k\}_k$ be the iterates generated by Algorithm \ref{alg-RSSM} with stepsizes {$\gamma_k >0 $} for all $k\ge0$. Then, for any $k\ge0$ and $Y \in \Stiefel n p$, we have
\begin{align*}
    \E\!&\textstyle\left[\,\|Y - X^{k+1}\|_{\mathcal A_{X^{k+1}}^{-1}}^2\,\middle|\,X^k\,\right] 
    \leq \textstyle\|Y-X^k\|_{\mathcal A_{X^k}^{-1}}^2\nonumber\\ 
    &\textstyle\qquad+ 2\gamma_k \left[f(Y)-f(X^k) + \frac{\tau+(2\ell-3)L}{\ell} \|Y-X^k\|_{\mathcal A_{X^k}^{-1}}^2\right]\\ 
    &\textstyle\qquad+ \gamma_k^2 L^2 \left(1+(\ell-2)[4\|Y-X^k\|+\|Y-X^k\|^2]\vphantom{\frac{1}{2}}\right)\\
    &\textstyle\qquad + (\ell-2) \gamma_k^3 L^3 (\|Y-X^k\|^2+3) + (\ell-2)\gamma_k^4 L^4.
\end{align*}
\end{lemma}

\begin{proof}
    Let $\{i,j\} \sim \mathrm{Uniform}\binom{[\ell]}{2}$ be drawn. Given $X^k \in \Stiefel n p$, we pick an arbitrary $\widetilde\nabla f(X^k) \in \pd f(X^k)$ and set $\widetilde\nabla_{\Stiefel n p} f(X) = \mathcal P_{T_{X}\Stiefel n p}(\widetilde\nabla f(X))$. {We have} $\left\|X_{ij}^{k+1} - X_{ij}^k + \gamma_k \widetilde \nabla^{C_{ij}}_{\mathcal{M}_{X_{-ij}^k}} f(X^k)\right\| \leq \gamma_k^2 \left\| \widetilde \nabla^{C_{ij}}_{\mathcal{M}_{X_{-ij}^k}} f(X^k)\right\|^2$ by the second-order boundedness condition \eqref{eq-secondbdd}. Denoting the identity operator on $\R^{n\times p}$ by $\mathcal I$, we bound
    \begin{align*}
    &\textstyle \|Y - X^{k+1}\|_{\mathcal A_{X^{k}}^{-1}}^2 = \frac{\ell}{2}\,\|Y - X^{k+1}\|^2 + \|Y - X^{k+1}\|_{\mathcal A_{X^{k}}^{-1}-\frac{\ell}{2}\mathcal I}^2\\
    &\textstyle\leq \frac{\ell}{2} \left\|Y-X^k+\gamma_k \widetilde \nabla^{C_{ij}}_{\mathcal{M}_{X_{-ij}^k}} f(X^k)I_{ij}^\top \right\|^2 + \Bigg(\left\|Y-X^k+\gamma_k \widetilde \nabla^{C_{ij}}_{\mathcal{M}_{X_{-ij}^k}} f(X^k)I_{ij}^\top\right\|_{\mathcal A_{X^k}^{-1}-\frac{\ell}{2}\mathcal I} \\
    &\textstyle\ \ + \|\mathcal A_{X^k}^{-1} - \frac{\ell}{2}\mathcal I\|_{\mathsf{op}}^{1/2} \cdot \left\|X^{k+1}-X^k+\gamma_k \widetilde \nabla^{C_{ij}}_{\mathcal{M}_{X_{-ij}^k}} f(X^k)I_{ij}^\top \right\|\Bigg)^2\\
    &\leq \left\|Y-X^k+\gamma_k \widetilde \nabla^{C_{ij}}_{\mathcal{M}_{X_{-ij}^k}} f(X^k)I_{ij}^\top \right\|_{\mathcal A_{X^k}^{-1}}^2\\
    &\textstyle\ \ + \ell(\ell-2) \gamma_k^2 \left\|\widetilde \nabla^{C_{ij}}_{\mathcal{M}_{X_{-ij}^k}} f(X^k)\right\|^2 \cdot \left\| Y-X^k+\gamma_k \widetilde \nabla^{C_{ij}}_{\mathcal{M}_{X_{-ij}^k}} f(X^k)I_{ij}^\top \right\| \\
    &\textstyle\ \ + \frac{\ell(\ell-2)}{2}\,\gamma_k^4 \left\|\widetilde \nabla^{C_{ij}}_{\mathcal{M}_{X_{-ij}^k}} f(X^k) \right\|^4\\
    &\leq \left\|Y-X^k+\gamma_k \widetilde \nabla^{C_{ij}}_{\mathcal{M}_{X_{-ij}^k}} f(X^k)I_{ij}^\top \right\|_{\mathcal A_{X^k}^{-1}}^2\\
    &\textstyle\ \ + 2(\ell-2) \gamma_k^2 \left\|\widetilde \nabla^{C_{ij}}_{\mathcal{M}_{X_{-ij}^k}} f(X^k)I_{ij}^\top\right\|_{\mathcal A_{X^k}^{-1}}^2 \cdot \left(\left\| Y-X^k\right\| +\gamma_kL \right) \\ 
    &\textstyle\ \ + (\ell-2)\gamma_k^4L^2 \left\|\widetilde \nabla^{C_{ij}}_{\mathcal{M}_{X_{-ij}^k}} f(X^k)I_{ij}^\top \right\|_{\mathcal A_{X^k}^{-1}}^2 \\ 
    &\leq \|Y-X^k\|_{\mathcal A_{X^k}^{-1}}^2 + 2\gamma_k \left\langle\,\widetilde \nabla^{C_{ij}}_{\mathcal{M}_{X_{-ij}^k}} f(X^k)I_{ij}^\top, Y-X^k\,\right\rangle_{\mathcal A_{X^k}^{-1}}\\
    &\textstyle\ \  + \gamma_k^2 \left\|\widetilde \nabla^{C_{ij}}_{\mathcal{M}_{X_{-ij}^k}}f(X^k)I_{ij}^\top\right\|_{\mathcal A_{X^k}^{-1}}^2 \left(1 + 2(\ell-2) \|Y-X^k\| \right) \\
    &\textstyle \ \ + 2(\ell-2)\gamma_k^3L \left\|\widetilde \nabla^{C_{ij}}_{\mathcal{M}_{X_{-ij}^k}} f(X^k)I_{ij}^\top\right\|_{\mathcal A_{X^k}^{-1}}^2 + (\ell-2)\gamma_k^4L^2 \left\|\widetilde \nabla^{C_{ij}}_{\mathcal{M}_{X_{-ij}^k}} f(X^k)I_{ij}^\top \right\|_{\mathcal A_{X^k}^{-1}}^2,
    \end{align*}
    where the first inequality follows from \eqref{eq-lipslike}; 
    the second and third inequalities are due to Lemma~\ref{lem-selfadjoint} and the fact that $\left\|\widetilde\nabla_{\mathcal M_{X_{-ij}^k}}^{C_{ij}}f(X^k)\right\| \leq L$. Conditioning on $X^k$ and taking expectation on both sides of the above inequality with respect to  $\{i,j\} \sim \mathrm {Uniform}\binom{[\ell]}{2}$, we get
    \begin{align*}
    &\E\!\left[\|Y - X^{k+1}\|_{\mathcal A_{X^{k}}^{-1}}^2\,\middle|\,X^k\right]\\
    &\textstyle\leq \|Y-X^k\|_{\mathcal A_{X^k}^{-1}}^2 + 2\gamma_k\,\langle\,\widetilde\nabla_{\Stiefel n p}f(X^k), Y-X^k\,\rangle \\
    &\textstyle\ \ + \gamma_k^2L^2(1 + 2(\ell-2) \|Y-X^k\|) + 2(\ell-2)\gamma_k^3L^3 + (\ell-2)\gamma_k^4L^4 \\
    &\textstyle\leq \|Y-X^k\|_{\mathcal A_{X^k}^{-1}}^2 + 2\gamma_k \left[f(Y)-f(X^k) + \frac{\tau+L}{2}\|Y-X^k\|^2\right]\\ 
    &\textstyle\ \ + \gamma_k^2 L^2(1 + 2(\ell-2) \|Y-X^k\|) + 2(\ell-2)\gamma_k^3 L^3 + (\ell-2)\gamma_k^4 L^4,
    \end{align*}
    where the first inequality is due to Lemma \ref{lem-cindep} and the fact that $\|\widetilde\nabla_{\Stiefel n p}f(X^k)\| \leq L$, and the second inequality is due to Lemma \ref{lem-RSI}. It then follows from \eqref{eq-almostNonexpansive} that
    \begin{align*}
    &\textstyle\E\!\left[\|Y - X^{k+1}\|_{\mathcal A_{X^{k+1}}^{-1}}^2\,\middle|\,X^k\right]\\
    &\textstyle\leq \|Y-X^k\|_{\mathcal A_{X^k}^{-1}}^2 + 2\gamma_k \left[f(Y)-f(X^k) + \frac{\tau+L}{2}\|Y-X^k\|^2\right]\\ 
    &\textstyle\quad+ \gamma_k^2 L^2(1 + 2(\ell-2) \|Y-X^k\|) + 2(\ell-2)\gamma_k^3 L^3 + (\ell-2)\gamma_k^4 L^4 \\
    &\textstyle\quad + 2 (\ell-2) \gamma_k L \|Y-X^k\|^2 + (\ell-2) \gamma_k^2 L^2\left(\|Y-X^k\|^2 + 2\|Y-X^k\|\right)\\ 
    &\textstyle\quad + (\ell-2) \gamma_k^3 L^3(\|Y-X^k\|^2+1) \\
    &\textstyle\leq \|Y-X^k\|_{\mathcal A_{X^k}^{-1}}^2 + 2\gamma_k \left[f(Y)-f(X^k) + \frac{\tau+(2\ell-3)L}{2} \cdot \frac{2}{\ell}\|Y-X^k\|_{\mathcal A_{X^k}^{-1}}^2\right]\\
    &\textstyle\quad + \gamma_k^2 L^2 \left(1+(\ell-2)[4\|Y-X^k\|+\|Y-X^k\|^2]\vphantom{\frac{1}{2}}\right)\\ 
    &\textstyle\quad + (\ell-2) \gamma_k^3 L^3 (\|Y-X^k\|^2+3) + (\ell-2)\gamma_k^4 L^4.
    \end{align*}
    Note that the bound above is independent of the choice of $\widetilde\nabla f(X^k) \in \pd f(X^k)$. This completes the proof.
\end{proof}


Lemma~\ref{lem-basicrecursion} allows us to prove the following inequality, which, upon rearranging, essentially establishes the sufficient decrease of the adaptive Moreau envelopes.

\begin{proposition}[Sufficient Decrease of $f^{\mathfrak C}_\lambda$] \label{prop-suffdec}
Let $\{X^k\}_k$ be the iterates generated by Algorithm~\ref{alg-RSSM} with stepsizes {$\gamma_k >0$} for all $k\ge0$. Then, for any $k\ge0$ and $\lambda \in (0, \frac{\ell}{2(\tau + (2\ell-1)L)} )$, conditioning on $X^k$, we have
\begin{align}
    \gamma_k \Theta_{\lambda f}^{\mathfrak C}(X^k)^2
    &= \frac{\gamma_k}{\lambda^2} \|P_{\lambda f}^{\mathfrak C}(X^k)-X^k\|_{\mathcal A_{X^k}^{-1}}^2\nonumber\\ 
    &\leq \frac{f_\lambda^{\mathfrak C}(X^k) - \E[\,f^{\mathfrak C}_\lambda(X^{k+1})\,|\,X^k\,] + \frac{9\gamma_k^2 L^2}{2\lambda} + \frac{2(\ell-2)\gamma_k^3 L^3}{\lambda} + \frac{(\ell-2)\gamma_k^4 L^4}{2\lambda}}{\frac{\lambda}{\ell} \left[\frac{\ell}{2\lambda} - (\tau+(2\ell-3) L)\right]}. \label{eq-suffdec}
\end{align}
\end{proposition}

\begin{proof}
    Applying Lemma~\ref{lem-basicrecursion} with $Y = P_{\lambda f}^{\mathfrak C}(X^k) \in \Stiefel n p$, we have
    \begin{align*}
    &\textstyle\E[\,f^{\mathfrak C}_\lambda(X^{k+1})\,|\,X^k\,]\textstyle\leq \E\left[f(P_{\lambda f}^{\mathfrak C}(X^k)) + \frac{1}{2\lambda}\|P_{\lambda f}^{\mathfrak C}(X^k)-X^{k+1}\|_{\mathcal A_{X^{k+1}}^{-1}}^2\,\middle|\,X^k\right]\\
    &\textstyle\leq f_\lambda^{\mathfrak C}(X^k) + \frac{\gamma_k}{\lambda} \left[f(P_{\lambda f}^{\mathfrak C}(X^k))-f(X^k) + \frac{\tau+(2\ell-3)L}{\ell}\|P_{\lambda f}^{\mathfrak C}(X^k)-X^k\|_{\mathcal A_{X^k}^{-1}}^2\right]\\ 
    &\textstyle\quad+ \frac{\gamma_k^2 L^2}{2\lambda} \left(1+(\ell-2) \left[ 4\left(\frac{4\lambda L}{\ell}\right)+\left(\frac{4\lambda L}{\ell}\right)^2 \right]\right) + \frac{(\ell-2) \gamma_k^3 L^3}{2\lambda} \left(\left(\frac{4\lambda L}{\ell}\right)^2+3\right) + \frac{(\ell-2)\gamma_k^4 L^4}{2\lambda}\\
    &\textstyle\leq f_\lambda^{\mathfrak C}(X^k) + \frac{\lambda}{\ell} \left(\tau + (2\ell-3) L - \frac{\ell}{2\lambda}\right) \cdot \gamma_k \Theta_{\lambda f}^{\mathfrak C}(X^k)^2\\
    &\textstyle\quad+ \frac{\gamma_k^2 L^2}{2\lambda} \left(1+(\ell-2) \frac{16\lambda L (\ell + \lambda L)}{\ell^2}\right) + \frac{(\ell-2) \gamma_k^3 L^3}{2\lambda} \cdot \frac{16(\lambda L)^2 + 3\ell^2}{\ell^{2}} + \frac{(\ell-2) \gamma_k^4 L^4}{2\lambda}\\
    &\textstyle< f_\lambda^{\mathfrak C}(X^k) + \frac{\lambda}{\ell} \left(\tau + (2\ell-3) L - \frac{\ell}{2\lambda}\right) \cdot \gamma_k \Theta_{\lambda f}^{\mathfrak C}(X^k)^2\\ 
    &\textstyle\quad+ \frac{\gamma_k^2 L^2}{2\lambda} \left(1+(\ell-2) \frac{4(2\ell+1)}{\ell^2}\right) + \frac{(\ell-2) \gamma_k^3 L^3}{2\lambda} \cdot \frac{4+3\ell^2}{\ell^2} + \frac{(\ell-2) \gamma_k^4 L^4}{2\lambda}\\
    &\textstyle\leq f_\lambda^{\mathfrak C}(X^k) + \frac{\lambda}{\ell} \left(\tau + (2\ell-3) L - \frac{\ell}{2\lambda}\right) \cdot \gamma_k \Theta_{\lambda f}^{\mathfrak C}(X^k)^2 + \frac{9\gamma_k^2 L^2}{2\lambda} + \frac{2(\ell-2)\gamma_k^3 L^3}{\lambda} + \frac{(\ell-2)\gamma_k^4 L^4}{2\lambda},
    \end{align*}
    where the second inequality holds because $\lambda \in (0, \frac{\ell}{2(\tau + (2\ell-1)L)})$ and $\|P_{\lambda f}^{\mathfrak C}(X^k)-X^k\| \leq \frac{4\lambda L}{\ell}$ {from the proof of Lemma~\ref{lem-proxlips}}, the third inequality follows from $f(P_{\lambda f}^{\mathfrak C}(X^k)) + \frac{1}{2\lambda} \|P_{\lambda f}^{\mathfrak C}(X^k)-X^k\|_{\mathcal A_{X^k}^{-1}}^2 \leq f(X^k)$, the fourth inequality is due to $\lambda < \frac{\ell}{2(\tau + (2\ell-1)L)} < \frac{\ell}{2(2\ell-1) L} < \frac{1}{2L}$, and the last inequality uses the fact that $\ell \ge 2$. 
    Upon rearranging and noting that $\frac{\ell}{2\lambda} - (\tau+(2\ell-3) L) > 0$, we obtain~\eqref{eq-suffdec}, as desired. 
    \end{proof}

By choosing suitable stepsizes $\{\gamma_k\}_k$ in Proposition~\ref{prop-suffdec} and using a telescoping argument, we obtain the following convergence rate of RSSM.


\begin{theorem}[Non-Asymptotic Convergence Rate] \label{thm-expectation}
Let $\lambda \in (0, \frac{\ell}{2(\tau + (2\ell-1)L)})$, $T\ge1$ be fixed and $\{X^k\}_{k=0}^{T}$ be the iterates generated by Algorithm~\ref{alg-RSSM} with the following stepsizes:
\begin{enumerate}[label=(\alph*)]
    \item \label{thm-expectation:a} If $\gamma_k = \frac{\Delta}{\ell^{1/3}(T+1)^{1/2}}$ for $k=0,1,\ldots,T$, where {$\Delta >0$} is any constant, then
\begin{align*}
&\min_{0 \leq k \leq T} \E[\Theta_{\lambda f}^{\mathfrak C}(X^k)^2]\\ 
&\leq \frac{f_\lambda^{\mathfrak C}(X^{0}) - \min\limits_{X \in \Stiefel n p}\,f_\lambda^{\mathfrak C}(X) + \frac{9L^2\Delta^2}{2\lambda \ell^{2/3}} + \frac{\ell-2}{\ell}\left(\frac{2L^3\Delta^3}{\lambda(T+1)^{1/2}} + {\frac{L^4\Delta^4}{{2}\lambda {\ell^{1/3}}(T+1)}}\right)}{\frac{\lambda}{\ell}\left(\frac{\ell}{2\lambda}- (\tau+(2\ell-3) L)\right) \cdot \frac{\Delta}{\ell^{1/3}}(T+1)^{1/2}}.
\end{align*}
{Moreover, by choosing $\lambda = \frac{\ell}{4(\tau + (2\ell-1)L)}$ and $T \geq \ell^{2/3}-1$, the iteration complexity (i.e., the number of iterations required to compute a point with the expected surrogate stationarity measure bounded by $\eps$) of RSSM is $\mathcal{O}(\ell^{2/3} \eps^{-4})$.
}

    \item\label{thm-expectation:b} If $\gamma_k = \frac{\Delta_k}{\ell^{1/3}(k+1)^{1/2}}$ with {$\Delta_k \in [\Delta_{\mathsf{min}}, \Delta_{\mathsf{max}}]$} for $k=0,1,\ldots,T$, where $\Delta_{\mathsf{min}}$, $\Delta_{\mathsf{max}} > 0$ are constants, then
\begin{align*}
&\min_{0 \leq k \leq T} \E[\Theta_{\lambda f}^{\mathfrak C}(X^k)^2] \\
& \leq \frac{ f_\lambda^{\mathfrak C}(X^{0}) - \min\limits_{X \in \Stiefel n p}\,f_\lambda^{\mathfrak C}(X) + {\frac{(9(1 + \log{(T+1)}) + {12}\ell^{1/3} + 2) \max\{ L^2 \Delta_{\mathsf{max}}^2,L^4 \Delta_{\mathsf{max}}^4 \}}{2\lambda\ell^{1/3}}} }{\frac{\lambda}{\ell}\left(\frac{\ell}{2\lambda} - (\tau+(2\ell-3) L)\right) \cdot \frac{\Delta_{\mathsf{min}}}{\ell^{1/3}} (T+1)^{1/2}}.
\end{align*}
\end{enumerate}
\end{theorem}

\begin{proof} 
Taking expectation on both sides of~\eqref{eq-suffdec}, we have
\[ 
    \gamma_k\, \E[\Theta_{\lambda f}^{\mathfrak C}(X^k)^2]
    \leq \frac{
    \E\!\left[f_\lambda^{\mathfrak C}(X^k)\right] - \E\!\left[f_\lambda^{\mathfrak C}(X^{k+1})\right] + \frac{9\gamma_k^2 L^2}{2\lambda} + \frac{2(\ell-2)\gamma_k^3 L^3}{\lambda} + \frac{(\ell-2)\gamma_k^4 L^4}{2\lambda}}{\frac{\lambda}{\ell} \left[\frac{\ell}{2\lambda} - (\tau+(2\ell-3) L)\right]}.
\] 
\newpage
Summing the above inequality over $k = 0, 1, \ldots, T$ gives
\begin{align}
&\min_{0 \leq k \leq T} \E[\,\Theta_{\lambda f}^{\mathfrak C}(X^k)^2\,] \nonumber\\
&\leq \frac{f_\lambda^{\mathfrak C}(X^{0}) - \E[\,f_\lambda^{\mathfrak C}(X^{T+1})] + \frac{9L^2}{2\lambda}\sum\limits_{k=0}^T \gamma_k^2 + \frac{2(\ell-2) L^3}{\lambda} \sum\limits_{k=0}^T \gamma_k^3 + \frac{(\ell-2)L^4}{2\lambda}\sum\limits_{k=0}^T \gamma_k^4}{\frac{\lambda}{\ell}\left(\frac{\ell}{2\lambda}- (\tau+(2\ell-3) L)\right) \sum\limits_{k=0}^T \gamma_k}. \label{eq-key3}
\end{align}

If $\gamma_k = \frac{\Delta}{\ell^{1/3}(T+1)^{1/2}}$ for $k=0,1,\ldots,T$, then $\sum_{k=0}^T \gamma_k = \frac{\Delta}{\ell^{1/3}}(T+1)^{1/2}$, $\sum_{k=0}^T \gamma_k^2 = \frac{\Delta^2}{\ell^{2/3}}$,  $\sum_{k=0}^T \gamma_k^3 = \frac{\Delta^3}{\ell(T+1)^{1/2}}$, and $\sum_{k=0}^T \gamma_k^4 = \frac{\Delta^4}{\ell^{4/3} (T+1)}$. Substituting these into~\eqref{eq-key3} gives the inequality in~\ref{thm-expectation:a}. Choosing $\lambda = \frac{\ell}{4(\tau + (2\ell-1)L)}$ and $T \geq {\ell^{2/3}}-1$ in the inequality of Theorem \ref{thm-expectation}\ref{thm-expectation:a} yields
\begin{align*}
   &\min_{0 \leq k \leq T} \E[\,\Theta_{\lambda f}^{\mathfrak C}(X^k)\,] \leq \sqrt{\frac{f_\lambda^{\mathfrak C}(X^{0}) - \min\limits_{X \in \Stiefel n p}\,f_\lambda^{\mathfrak C}(X) + \frac{9L^2\Delta^2}{2\lambda \ell^{2/3}} + \frac{\ell-2}{\ell^{4/3}}\left(\frac{2L^3\Delta^3}{\lambda} + \frac{L^4\Delta^4}{{2\lambda \ell^{2/3}}}\right)}{\frac{\lambda}{\ell}\left(\frac{\ell}{2\lambda} - \frac{\ell}{4\lambda}\right) \cdot \frac{\Delta}{\ell^{1/3}} (T+1)^{1/2}}}, 
\end{align*}
which implies that the iteration complexity of RSSM is $\mathcal{O}(\ell^{2/3} \eps^{-4})$. {If, instead, $\gamma_k = \frac{\Delta_k}{\ell^{1/3}(k+1)^{1/2}}$ with $\Delta_k \in [\Delta_{\mathsf{min}}, \Delta_{\mathsf{max}}]$ for $k=0,1,\ldots,T$, then $\sum_{k=0}^T \gamma_k^2 \leq \frac{\Delta^2_{\mathsf{max}}}{\ell^{2/3} }\left(1 + \int_1^{{T+1}} \frac{du}{u}\right) = \frac{\Delta^2_{\mathsf{max}}}{\ell^{2/3}}(1 + \log {(T+1)})$, $\sum_{k=0}^T \gamma_k^3 \leq \frac{\Delta^3_{\mathsf{max}}}{\ell} \left(1+\int_1^{\infty} \frac{du}{u^{3/2}}\right) = \frac{{3}\Delta^3_{\mathsf{max}}}{\ell}$, and $\sum_{k=0}^T \gamma_k^4 \leq \frac{\Delta^4_{\mathsf{max}}}{\ell^{4/3} }\left(1 + \int_1^{\infty} \frac{du}{u^2}\right) = \frac{2\Delta^4_{\mathsf{max}}}{\ell^{4/3}}$. Moreover, we have $\sum_{k=0}^T \gamma_k \ge (T+1) {\frac{\Delta_{\mathsf{min}}}{\ell^{1/3}(T+1)^{1/2}} } = \frac{\Delta_{\mathsf{min}}}{\ell^{1/3}} (T+1)^{1/2}$. Substituting these into~\eqref{eq-key3} {and using $a^3 \le \max\{a^2,a^4\}$ for $a\ge 0$} give~\ref{thm-expectation:b}. }
\end{proof}

{Theorem \ref{thm-expectation}\ref{thm-expectation:a}, together with Remark \ref{remark:per-iter-cost} on the per-iteration complexity of RSSM, implies that the overall complexity of RSSM is
\[\mathcal{O}\left(\left(\chi(f)\cdot\alpha(\ell)+ {np^2}/{\ell}\right) \cdot \ell^{2/3} \eps^{-4} \right).
\]
In particular, for the orthogonal dictionary learning and sparse principal component analysis applications considered in Section \ref{sect-exp}, the above bound simplifies to
$\mathcal{O}((mn^2 + n^3)\cdot \ell^{-1/3} \varepsilon^{-4})$ and $\mathcal{O}((\min\{m,n\}\cdot np + np^2)\cdot \ell^{-1/3} \varepsilon^{-4})$, respectively.}

{\begin{remark}\label{remark: other_retraction}
    For general retractions on the Stiefel manifold (e.g., QR decomposition and Cayley transformation), the second-order boundedness property~\eqref{eq-secondbdd} still holds but with a possibly different constant $b\ge0$ taking the place of $\tfrac{1}{2}$; see~\cite[Appendix B]{boumal2019global}. This yields the following analog of~\eqref{eq-lipslike}: For all $X,Y \in \Stiefel n p$ and $\xi \in T_X \Stiefel n p$, we have
    \begin{align*}
    \| {\rm Retr}_X(\xi) - Y \| \le \| X + \xi - Y \| + b \| \xi \|^2;
    \end{align*}
    see~\cite[Section 6]{li2021weakly}. Suppose that we use the QR decomposition-based retraction $\mathcal{R}_{\Stiefel n {p_{ij}}}(\cdot)$ in \eqref{eq-projsubmfdblk}. Then, we have 
    \begin{align*}
        \mathcal{R}_{\Stiefel n {p_{ij}}}\left( X_{ij}^k - \gamma_k \widetilde \nabla^{C_{ij}}_{\mathcal{M}_{X_{-ij}^k}} f(X^k) \right) = Q^k \ \text{ with } \  Q^k R^k = X_{ij}^k - \gamma_k \widetilde \nabla^{C_{ij}}_{\mathcal{M}_{X_{-ij}^k}} f(X^k).
    \end{align*}
    Since $Q^k$ has the same column space as $X_{ij}^k - \gamma_k \widetilde \nabla^{C_{ij}}_{\mathcal{M}_{X_{-ij}^k}} f(X^k)$, we know that $Q^k \in \mathcal{M}_{X_{-ij}^k}$. This allows us to define an associated retraction on the submanifold block $\mathcal{M}_{X_{-ij}}$ for every $X_{-ij}$. In this case, the conclusions of Lemma \ref{lem-almostNonexpansive}, Lemma \ref{lem-basicrecursion}, Proposition \ref{prop-suffdec}, and Theorem \ref{thm-expectation} continue to hold after minor modifications to the proofs using the aforementioned analogs of~\eqref{eq-lipslike} and~\eqref{eq-secondbdd}. However, if we use a general retraction in~\eqref{eq-projsubmfdblk}, then the equality there may not hold, and it is not clear what the retraction on the submanifold block $\mathcal{M}_{X_{-ij}^k}$ should be. 
\end{remark}}



Theorem~\ref{thm-expectation} shows that with sufficiently many iterations, our proposed RSSM will produce an iterate $\bar{X}$ such that the expectation $\E[\,\Theta_{\lambda f}^{\mathfrak C}(\bar{X})^2\,]$ is smaller than any given threshold. However, it does not imply that the random quantity $\Theta_{\lambda f}^{\mathfrak C}(\bar{X})$ will be small with high probability or that the sequence $\{\E[\,\Theta_{\lambda f}^{\mathfrak C}(X^k)^2\,]\}_k$ converges. To complement Theorem~\ref{thm-expectation}, we now demonstrate the almost-sure convergence behavior of RSSM. To begin, let $\Lambda = \left\{ (\omega_1,\omega_2,\ldots) : \omega_i \in \binom{[\ell]}{2} \mbox{ for } i=1,2,\ldots \right\}$ denote the set of all possible sequences of pairs in $\binom{[\ell]}{2}$ generated by Algorithm~\ref{alg-RSSM}. It is possible to construct a $\sigma$-algebra $\mathscr{F}$ on $\Lambda$ and a probability measure $\mathbb{P}$ on $\mathscr{F}$ based on the uniform distribution on $\binom{[\ell]}{2}$ such that the triple $(\Lambda,\mathscr{F},\mathbb{P})$ is a probability space; see, e.g., the subsection ``Sequence Space'' in Chapter 1, Section 2 of~\cite{billingsley1995probability}. This allows us to study various events associated with the (random) sequence of iterates generated by Algorithm~\ref{alg-RSSM}.

\begin{theorem}[Almost-Sure Convergence] \label{thm-asympconv}
Let $\{X^k\}_k$ be the iterates generated by Algorithm~\ref{alg-RSSM} with stepsizes $\{\gamma_k\}_k$ satisfying {$\gamma_k>0$} for all $k\ge0$, $\sum_{k\ge0} \gamma_k = \infty$, and $\sum_{k\ge0} \gamma_k^2  < \infty$. Then, for any $\lambda \in (0, \frac{\ell}{2(\tau + (2\ell-1)L)})$, we will have $\lim\limits_{k\to \infty} \Theta_{\lambda f}^{\mathfrak C}(X^k) = 0$ $\mathbb{P}$-almost surely. Hence, every accumulation point of $\{X^k\}_k$ will be a stationary point of Problem~\eqref{mainpb} $\mathbb{P}$-almost surely.
\end{theorem}

\begin{proof} 
The proof involves verifying the conditions in Lemma~\ref{convlem} with $Y^k=X^k$, $\mu_k=\gamma_k$ for $k\ge0$ and $\Phi = \Theta_{\lambda f}^{\mathfrak C}$. First, using Lemma~\ref{lem-proxlips} with $X=X^k$ and $X^+=X^{k+1}$ in~\eqref{eq-lipslike2} and the definition of $\Theta_{\lambda f}^{\mathfrak C}$, we claim that $|\Theta_{\lambda f}^{\mathfrak C}(X^{k+1}) - \Theta_{\lambda f}^{\mathfrak C}(X^{k})| \leq L_\Theta \|X^{k+1}-X^k\|$, where $L_\Theta > 0$ is independent of $X^k$. {
Indeed, for any $X\in\Stiefel{n}{p}$, Lemma~\ref{lem-selfadjoint} gives
$\|\mathcal A_X^{-1/2}\|_{\mathsf{op}}
\leq \sqrt{\tbinom{\ell}{2}}$, while the identity~\eqref{eq-scaling-coordrep} gives
\(
\bigl(\mathcal A_X(\xi)\bigr)_j = \bigl(\mathcal A_X(\xi)\bigr)_{C_j}
=
\frac{2}{\ell}\bigl(I-X_{-j}X_{-j}^\top\bigr)\xi_j
+{\binom{\ell}{2}}^{-1}X_{-j}X_{-j}^\top\xi_j
\)
for $j \in [\ell]$. In particular, the linear operator $\mathcal A_X$ can be viewed as comprising $\ell$ blocks, where the $j$-th block $(j \in [\ell]$) acts as multiplication by
$\binom{\ell}{2}^{-1}$ on $\operatorname{range}(X_{-j})$
and by $\tfrac{2}{\ell}$ on its orthogonal complement. This implies that
\[
\begin{aligned}
\bigl(\mathcal A_X^{-1/2}(\xi)\bigr)_j
&=
\sqrt{\tfrac{\ell}{2}}
\bigl(I-X_{-j}X_{-j}^\top\bigr)\xi_j
+\sqrt{\tbinom{\ell}{2}}X_{-j}X_{-j}^\top\xi_j\\
&=
\sqrt{\tfrac{\ell}{2}}\,\xi_j
+
\left(\sqrt{\tbinom{\ell}{2}}-\sqrt{\tfrac{\ell}{2}}\right)
X_{-j}X_{-j}^\top\xi_j
\end{aligned}
\]
for $j\in[\ell]$. By Lemma \ref{lem-diffinprojections}, for any $X,Y \in \Stiefel{n}{p}$, we have
\[
\begin{aligned}
\|Y_{-j}Y_{-j}^\top-X_{-j}X_{-j}^\top\|_{\mathsf{op}}
&=\|(I-X_{-j}X_{-j}^\top)Y_{-j}\|_{\mathsf{op}}=\|(I-X_{-j}X_{-j}^\top)(Y_{-j}-X_{-j})\|_{\mathsf{op}}\\
&\leq \|Y_{-j}-X_{-j}\|_{\mathsf{op}}\leq \|Y_{-j}-X_{-j}\|_F
\leq \|Y-X\|_F.
\end{aligned}
\]
Writing
$\beta_\ell=\sqrt{\tbinom{\ell}{2}}-\sqrt{\tfrac{\ell}{2}}$,
we therefore obtain, for any $\xi\in\mathbb R^{n\times p}$,
\begin{align*}
&\quad \bigl\|
(\mathcal A_Y^{-1/2}-\mathcal A_X^{-1/2})(\xi)
\bigr\|^2\\
&=
\beta_\ell^2\sum_{j=1}^{\ell}
\bigl\|
(Y_{-j}Y_{-j}^\top-X_{-j}X_{-j}^\top)\xi_j
\bigr\|^2\leq
\beta_\ell^2\|Y-X\|^2
\sum_{j=1}^{\ell}\|\xi_j\|^2=
\beta_\ell^2\|Y-X\|^2 \cdot \|\xi\|^2.
\end{align*}
It then follows that  
\begin{equation} \label{eq:inv-op-bd}
\|\mathcal A_Y^{-1/2}-\mathcal A_X^{-1/2}\|_{\mathsf{op}}
\leq
\left(\sqrt{\tbinom{\ell}{2}}-\sqrt{\tfrac{\ell}{2}}\right)
\|Y-X\|.    
\end{equation}
Now, we compute
\begin{align*}
&\lambda\left|
\Theta_{\lambda f}^{\mathfrak C}(X^{k+1})
-\Theta_{\lambda f}^{\mathfrak C}(X^k)
\right|\\
&\leq
\Big\|
\mathcal A_{X^{k+1}}^{-1/2}
\bigl(P_{\lambda f}^{\mathfrak C}(X^{k+1})-X^{k+1}\bigr)
-
\mathcal A_{X^k}^{-1/2}
\bigl(P_{\lambda f}^{\mathfrak C}(X^k)-X^k\bigr)
\Big\|\\
&\leq
\Big\|
\mathcal A_{X^{k+1}}^{-1/2}
\bigl(
P_{\lambda f}^{\mathfrak C}(X^{k+1})
-P_{\lambda f}^{\mathfrak C}(X^k)
-X^{k+1}+X^k
\bigr)
\Big\|\\
&\quad+
\Big\|
\bigl(
\mathcal A_{X^{k+1}}^{-1/2}
-\mathcal A_{X^k}^{-1/2}
\bigr)
\bigl(P_{\lambda f}^{\mathfrak C}(X^k)-X^k\bigr)
\Big\|\\
&\leq
\sqrt{\tbinom{\ell}{2}}
\left(
\|P_{\lambda f}^{\mathfrak C}(X^{k+1})
-P_{\lambda f}^{\mathfrak C}(X^k)\|
+\|X^{k+1}-X^k\|
\right)\\
&\quad+
\|\mathcal A_{X^{k+1}}^{-1/2}
-\mathcal A_{X^k}^{-1/2}\|_{\mathsf{op}}
\cdot \|P_{\lambda f}^{\mathfrak C}(X^k)-X^k\|\\
&\leq
\sqrt{\tbinom{\ell}{2}}
\left(
\frac{\sqrt{2}\tbinom{\ell}{2}}
{\tfrac{\ell}{2}-\lambda(\tau+(2\ell-1)L)}
+1
\right)\|X^{k+1}-X^k\|\\
&\quad+
\tfrac{4\lambda L}{\ell}
\left(\sqrt{\tbinom{\ell}{2}}-\sqrt{\tfrac{\ell}{2}}\right)
\|X^{k+1}-X^k\|,
\end{align*}
where the last inequality uses~\eqref{eq-lipslike2},~\eqref{eq:inv-op-bd}, and the fact that the proof of Lemma~\ref{lem-proxlips} gives
$\|P_{\lambda f}^{\mathfrak C}(X^k)-X^k\| \leq \tfrac{4\lambda L}{\ell}$. This establishes the claim with
$L_\Theta>0$ depending only on $\ell$, $\lambda$,
$\tau$, and $L$.
} Next, we have $\|X^k - X^{k+1}\| \leq L\gamma_k$ for all $k\ge0$ and $\sum_{k\ge0} \gamma_k = \infty$. Lastly, let $f^\ast = \min_{X \in \Stiefel n p} f_\lambda^{\mathfrak C}(X) = \min_{X \in \Stiefel n p} f(X)$. {By Proposition \ref{prop-suffdec}}, we have
\begin{align*}
&\textstyle\E\left[\,f_\lambda^{\mathfrak C}(X^{k+1}) - f^\ast\,\middle|\,X^k\,\right]\\ 
&\textstyle\leq
\left(f_\lambda^{\mathfrak C}(X^k) - f^\ast\right) - \frac{\lambda}{\ell}\left(\frac{\ell}{2\lambda}-(\tau+(2\ell-3) L)\right) \cdot\,\gamma_k \Theta_{\lambda f}^{\mathfrak C}(X^k)^2\\
&\textstyle\quad + \frac{9\gamma_k^2 L^2}{2\lambda} + \frac{2(\ell-2)\gamma_k^3 L^3}{\lambda} + \frac{(\ell-2)\gamma_k^4 L^4}{2\lambda}.
\end{align*}
Since $\sum_{k\ge0} \gamma_k^2 < \infty$, we have $\sum_{k\geq 0} \gamma_k^r < \infty$ for $r \geq 2$. Thus, by the supermartingale convergence theorem~\cite[Theorem 1]{robbins1971convergence}, 
we will have $\sum_{k\ge0} \gamma_k \Theta_{\lambda f}^{\mathfrak C} (X^k)^2 < \infty$ $\mathbb{P}$-almost surely.
By Lemma \ref{convlem}, we will have $\lim\limits_{k \to \infty} \Theta_{\lambda f}^{\mathfrak C}(X^k) = 0$ $\mathbb{P}$-almost surely. 


{Let $\{ X^{k_m} \}_m\to\bar X$ be any convergent subsequence of $\{ X^k \}_k$ and set
$Z^m:=P_{\lambda f}^{\mathfrak C}(X^{k_m})$ for $m\ge1$.
By Lemma~\ref{lem-selfadjoint} and the definition of
$\Theta_{\lambda f}^{\mathfrak C}$, we have
$\|Z^m-X^{k_m}\| \le \sqrt{\frac{2}{\ell}} \|Z^m-X^{k_m}\|_{\mathcal A_{X^{k_m}}^{-1}}=\lambda\sqrt{\frac{2}{\ell}}\,
\Theta_{\lambda f}^{\mathfrak C}(X^{k_m}) \to 0$ and hence $Z^m\to\bar X$.
Moreover, Proposition~\ref{prop-stationarymeasure}\ref{prop-stationarymeasure:a}
ensures that there exist $V^m\in\pd_{\Stiefel n p}f(Z^m)$ for $m\ge1$ with $V^m\to0$.
By~\eqref{eq-Rsubdiff}, we can write
$V^m=\mathcal P_{T_{Z^m}\Stiefel n p}(G^m)$
for some $G^m\in\pd f(Z^m)$.
Since $\|G^m\|\leq L$, by passing to a further subsequence if necessary, we have $G^m\to\bar G$.
The outer semicontinuity of the Clarke subdifferential
\cite[Proposition 2.1.5(b)]{clarke1990optimization}
then yields $\bar G\in\pd f(\bar X)$.
By the continuity of $(X,G)\mapsto\mathcal P_{T_X\Stiefel n p}(G) = G - X\Sym{X^\top G}$, we have
\[
\mathcal P_{T_{\bar X}\Stiefel n p}(\bar G)
=\lim_{m\to\infty}
\mathcal P_{T_{Z^m}\Stiefel n p}(G^m)
=\lim_{m\to\infty}V^m=0,\quad \text{i.e.},\ 0 \in \pd_{\Stiefel n p} f(\bar X).
\]}Thus, every accumulation point of $\{X^k\}_k$ will be a stationary point of Problem~\eqref{mainpb}
$\mathbb P$-almost surely.
\end{proof}

With a slightly more conservative choice of stepsizes, we obtain the following almost-sure asymptotic convergence rate result, which complements the in-expectation non-asymptotic convergence rate results in Theorem~\ref{thm-expectation}.

\begin{theorem}[Almost-Sure Asymptotic Convergence Rate] \label{thm-convrate} 
Let {$\Delta_{\mathsf{min}}, \Delta_{\mathsf{max}} > 0$ be given constants} and $\{\Delta_k\}_k$ be a sequence in {$[\Delta_{\mathsf{min}}, \Delta_{\mathsf{max}}]$}.
Let $\{X^k\}_k $ be the sequence of iterates generated by Algorithm~\ref{alg-RSSM} with stepsizes $\gamma_k = \frac{\Delta_k}{\sqrt{k+2} \log(k+2)}$ for all $k\ge0$. Then, for any $\lambda \in (0, \frac{\ell}{2(\tau + (2\ell-1)L)})$, we will have $\liminf\limits_{k \to \infty} \sqrt[4]{k+2} \cdot \Theta_{\lambda f}^{\mathfrak C}(X^k) = 0$ $\mathbb{P}$-almost surely.  
\end{theorem}

\begin{proof} 
Observe that
{$\sum_{k=0}^K \gamma_k^2 \leq \gamma_0^2 + \int_0^K \frac{\Delta_{\mathsf{max}}^2}{(u+2) \log^2(u+2)}\,du \leq \gamma_0^2 + \frac{\Delta_{\mathsf{max}}^2}{\log 2} < {3\Delta_{\mathsf{max}}^2}$}
and $\sum_{k=0}^K \gamma_k \geq \sum_{k=0}^K \frac{\Delta_{\mathsf{min}}}{ \sqrt{K+2} \log(K+2)} = \frac{\Delta_{\textsf{min}} (K+1)}{\sqrt{K+2} \log(K+2)} \to +\infty$ as $K \to +\infty$. Let $\Lambda_\delta = \{ \omega \in \Lambda : \liminf_{k \to \infty} \sqrt{k+2} \cdot \Theta_{\lambda f}^{\mathfrak C}(X^k(\omega))^2 \geq \delta \}$ and assume to the contrary that $\mathbb P(\Lambda_\delta) > 0$ for some $\delta > 0$. {Then, there exists a $\bar k\ge0$ such that $\mathbb P(\,\bigcap_{k \geq \bar k} \Lambda_{\delta, k}\,) > 0$, where $\Lambda_{\delta, k} = \{ \omega \in \Lambda : \sqrt{k+2} \cdot \Theta_{\lambda f}^{\mathfrak C}(X^k(\omega))^2 \geq \delta/2 \}$.} This implies that
\begin{align*}
\textstyle
\mathbb P\left(\,\left\{\,\omega \in \Lambda : \sum_{k \geq \bar k} \frac{\Theta_{\lambda f}^{\mathfrak C}(X^k(\omega))^2}{\sqrt{k+2}\,\log(k+2)} \geq \sum_{k \geq \bar k} \frac{\delta /2}{(k+2) \log(k+2)}\,\right\}\,\right) \geq \mathbb P\left(\bigcap_{k\geq \bar k} \Lambda_{\delta, k}\right) > 0.
\end{align*}
Note that $\sum_{k \geq \bar k} \frac{1}{(k+2) \log (k+2)} \geq \int_{\bar k + 3}^\infty \frac{du}{u \log u} = \infty$. However, as shown in the proof of Theorem~\ref{thm-asympconv}, we will have $\sum_{k\ge0} \gamma_k \Theta_{\lambda f}^{\mathfrak C}(X^k(\omega))^2 < \infty$ $\mathbb{P}$-almost surely. Thus, we obtain a contradiction.

\end{proof}

\section{Numerical Results} \label{sect-exp}
In this section, we present numerical results to demonstrate the performance of our proposed RSSM on the 
dual principal component pursuit (DPCP) formulation of robust subspace recovery (RSR), on orthogonal dictionary learning (ODL), {and on sparse principal component analysis (Sparse PCA)}. We compare the performance of RSSM against the Riemannian subgradient method (RSM) \cite{li2021weakly}. {Furthermore, in Section~\ref{sect-spca}, we evaluate RSSM for Sparse PCA against both RSM and the ManPG algorithm introduced in \cite{chen2020proximal}.}

\subsection{DPCP Formulation of Robust Subspace Recovery}

The DPCP formulation of robust subspace recovery aims to identify a low-dimensional linear subspace for a dataset corrupted by outliers, which plays a fundamental role in machine learning, signal processing, and statistics. In robust subspace recovery, one is given some measurements $\widetilde Y = [\,Y\ O\,]\,\Gamma \in \mathbb{R}^{n \times m}$, where the columns of $Y \in \mathbb{R}^{n \times m_1}$ form inlier points spanning a $d$-dimensional subspace $\mathcal{S}$, the columns of $O \in \mathbb{R}^{n \times m_2}$ form outlier points with no linear structure, $\Gamma \in \mathbb{R}^{m \times m}$ is an unknown permutation, and $m = m_1 + m_2$.
To recover the subspace $\mathcal{S}$ via obtaining an orthonormal basis for $\mathcal{S}^\perp$, one could adopt the \emph{holistic approach}~\cite{ding2021dual} and solve
\begin{align} \label{eq-RSR}
    \min_{X \in \mathbb{R}^{n \times (n-d)}} \; f(X) = \frac{1}{m} \sum_{j=1}^m \|\tilde y_j^\top X\|\ \text{ s.t. }\ X \in \Stiefel n {n-d}.
\end{align}
\subsubsection{Synthetic Data} In our experiments, we consider the setting {$n=600$, $p=n-d=540$} and generate a {uniformly distributed random matrix $S \in \Stiefel n {d}$ to obtain the subspace $\mathcal S = \mathop{\mathrm{span}} S$.} We also generate $m_1 = 1500$ inlier points by $Y = SR$ {with $R \in \R^{d \times {m_1}}$ and $m_2 = 3500$ outlier points $O \in \R^{n\times m_2}$, in which all entries are i.i.d. standard Gaussian.} For both RSM and RSSM, we use the same initialization 
$X^0 = \argmin_{X \in \Stiefel n p} 
\| \widetilde Y^\top X \|^2$, where $\widetilde Y = [\,Y\ O\,]\,\Gamma$ for some random permutation matrix $\Gamma$. We partition the columns into {$16$, $30$, $60$, $80$ blocks (i.e., $\ell = 16, 30, 60, 80$) and adopt the stepsizes $\gamma_k = \frac{0.45 \eta^{2(2(0.9995)^{k\theta}-1)}}{\sqrt{k\theta+2}\log(k\theta+2)}$ with $\eta = \min\{\ell,40\}$, $\theta = \binom{\eta}{2}/\binom{\ell}{2}$ for RSSM, which includes RSM when $\ell=2$.} 

\begin{figure}[ht]
\begin{minipage}[t]{0.5\linewidth} 
\centering 
\includegraphics[scale=0.40]{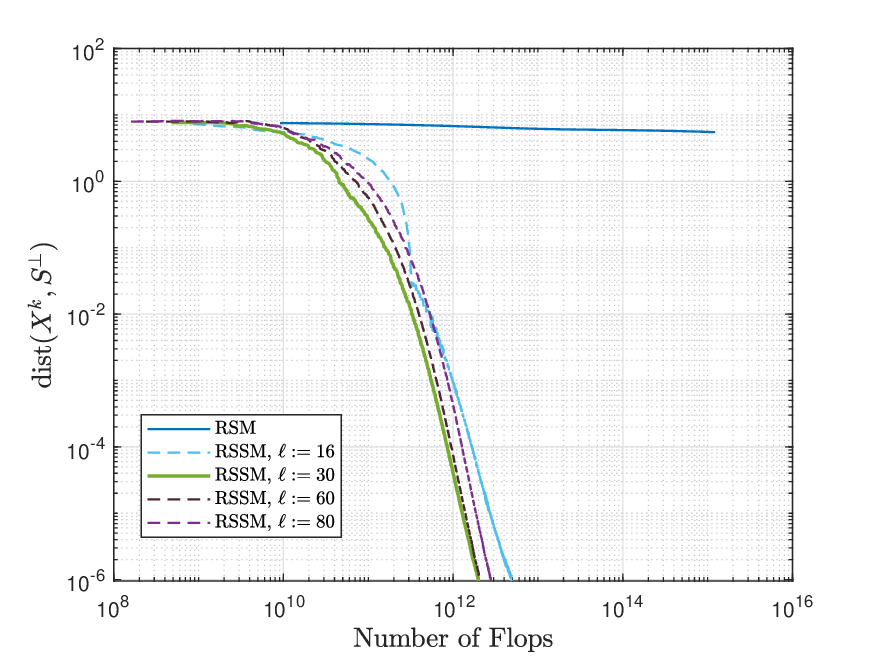} 
\end{minipage}%
\begin{minipage}[t]{0.5\linewidth} 
\centering 
\includegraphics[scale=0.40]{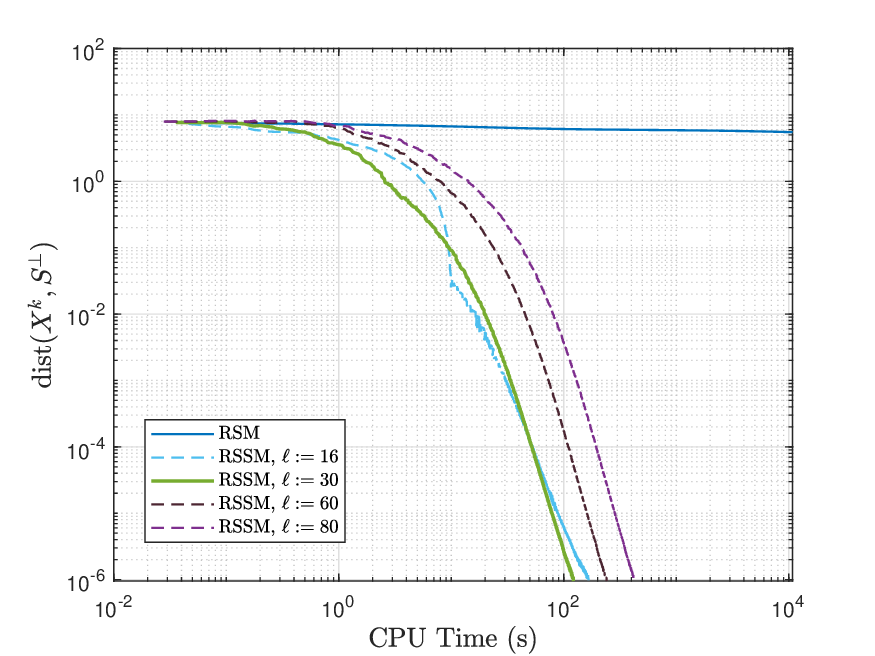}
\end{minipage} 
\vspace{-0.5cm}
\caption{{Robust subspace recovery: $\gamma_k = \frac{0.45 \eta^{2(2(0.9995)^{k\theta}-1)}}{\sqrt{k\theta+2}\log(k\theta+2)}$ at $k$-th iteration with $\eta = \min\{\ell,40\}$ and $\theta = \binom{\eta}{2}/\binom{\ell}{2}$.}}
\label{fig-rsr1}
\end{figure}

We measure the performance by the \emph{estimation error}
\begin{align*}
    \dist(X,S^\perp) &= \min_{R \in \Orth p} \|X-S^\perp R\| = \sqrt{2(p-\|X^\top S^\perp\|_{\ast})}  = \sqrt{2(p-\|(I-SS^\top)X\|_{\ast})},
\end{align*}
where $\Orth p = \Stiefel p p$.
Figure~\ref{fig-rsr1} shows the log-log plots of estimation error versus flops and CPU time. {We observe that, despite its (seeming) convergence, the accuracy attained by RSM is unsatisfactory, reaching only the order of $10^0$ in 3 hours. By contrast, RSSM outperforms RSM in terms of both accuracy and speed. Among the tested settings in our numerical experiments, RSSM achieves its best performance when $\ell = 30$, reaching an accuracy on the order of $ 10^{-6}$ in approximately 2 minutes.}
\subsection{Orthogonal Dictionary Learning}\label{subsec: odl}

ODL aims to decompose a data matrix $Y \in \R^{n \times m}$ into the multiplication of an orthonormal dictionary and a sparse matrix, i.e., $Y = XS$ with $X\in \Orth n = \Stiefel n n$ and each column of $S \in \R^{n \times m}$ being sparse. To recover the dictionary $X$, a natural optimization formulation is 
\begin{align*} \min_{X \in \R^{n \times n}} \|Y^\top X \|_1 \ \text{ s.t. } \ X \in \Orth n.
\end{align*}
\subsubsection{Synthetic Data} 
\begin{figure}[ht]
\begin{minipage}[t]{0.5\linewidth} 
\centering 
\includegraphics[scale=0.40]{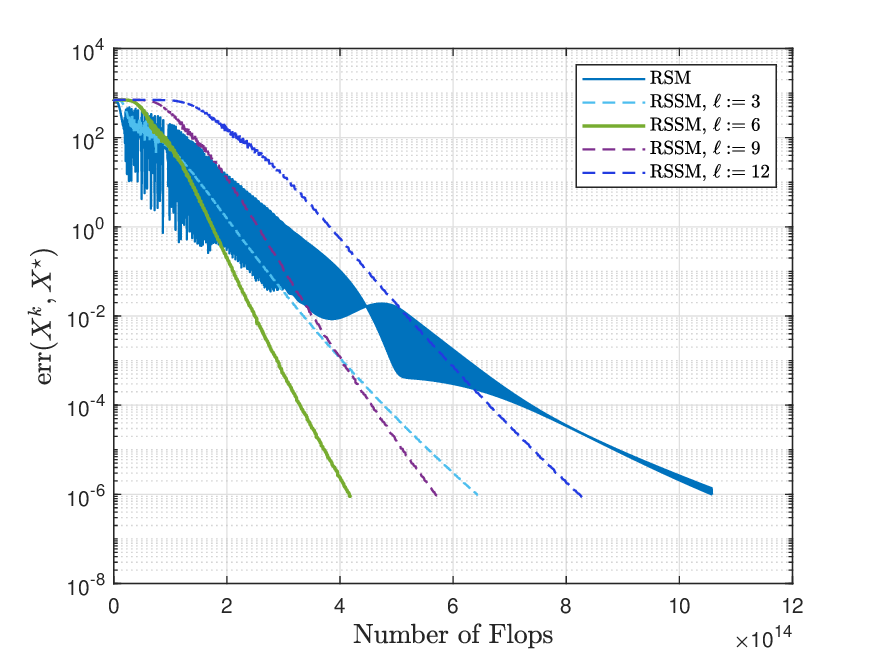} 
\end{minipage}%
\begin{minipage}[t]{0.5\linewidth} 
\centering 
\includegraphics[scale=0.40]{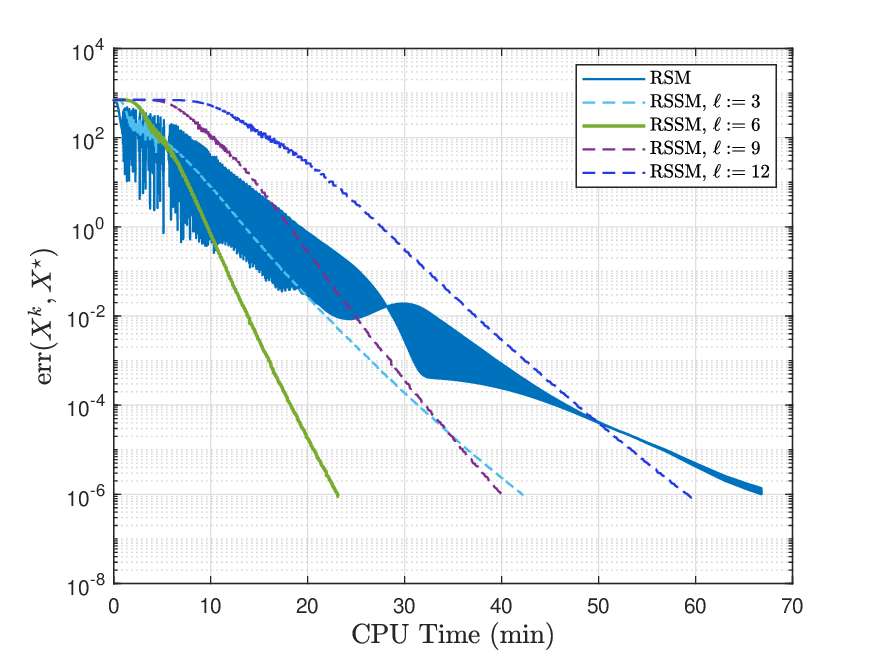}
\end{minipage} 
\vspace{-0.5cm}
\caption{{Orthogonal dictionary learning: $\gamma_k = \frac{10^{-3} \eta^{22(0.9995)^{k\theta}-20}}{\sqrt{k\theta+2}\log(k\theta+2)}$ at $k$-th cycle with $\eta = \min\{\ell,6.5\}$ and $\theta = \min\{\ell-1,6\}$.}}
\label{fig-odl1}
\end{figure}We generate synthetic data as in \cite{bai2018subgradient, li2021weakly}. First, the true orthogonal dictionary $X^\star \in \Orth n$ with $n ={800}$, is generated randomly according to the uniform distribution. The sample size is $m ={226274} \approx 10 \times n^{1.5}$. Then we generate a sparse matrix $S \in \R^{n \times m}$ with i.i.d. entries $S_{ij} \sim \mathcal N(0,1)$ with probability $0.3$ and $S_{ij} = 0$ otherwise. The observation is thus $Y = X^\star S$. Next, we generate a uniformly random orthogonal matrix $X^0$ on $\Stiefel n n$ to initialize the algorithms. We partition the columns into $3$, $6$, $9$, and $12$ blocks, i.e., $\ell = 3, 6, 9, 12$. To prevent the updates from becoming numerically negligible, we set the stepsize for each cycle $k$ to {$\gamma_k = \frac{10^{-3}\eta^{22(0.9995)^{k\theta}-20}}{\sqrt{k\theta+2}\log(k\theta+2)}$ with $\eta = \min\{\ell,6.5\}$ and $\theta = \min\{\ell-1,6\}$}. Here, one cycle comprises $\binom{\ell}{2}$ consecutive iterations for RSSM {with $\ell$ column blocks, which naturally includes RSM when $\ell = 2$}. 

To evaluate the performance of the algorithms, we define the error between $X$ and $X^\star$ as $\mathrm{err}(X, X^\star) = \sum_{i=1}^n \left|\max_{1\leq j\leq n} |x_i^\top x^\star_{j}|-1\right|$, where $x_i$ and $x_j^\star$ denote the $i$-th and $j$-th columns of $X$ and $X^\star$ respectively. Figure \ref{fig-odl1} shows the plots of error in log scale against flops and CPU time. {Once again, we observe that RSSM outperforms RSM in terms of both speed and accuracy. In this experiment setting, RSSM attains its best performance when $\ell = 6$, converging to a solution with an accuracy of $\mathrm{err}(X,X^\star) \leq 10^{-6}$ in approximately 23 minutes. This CPU time is significantly shorter than the approximately $66$ minutes required by RSM.}

\subsection{{Sparse Principal Component Analysis}} \label{sect-spca}

{
The goal of sparse PCA is to find a set of new uncorrelated variables formed by \emph{sparse} linear combination of the original variables, called \emph{principal components}, that can explain the greatest amount of variance from the data; see, e.g., \cite{cunningham2015linear}. We adopt the following formulation in~\cite{chen2020proximal}, which minimizes the \(\ell_1\)-regularized negative variance and simultaneously enforces sparsity and orthogonality \cite{journee2010generalized}:
\begin{align} \label{eq-sPCA}
 \min_{X \in \mathbb{R}^{n \times p}} \; f(X) = - \frac{1}{2m} \|A X\|^2 +\mu \|X\|_1 \ \text { s.t. } \ X \in \Stiefel n p.
\end{align}
Here, $A := [a_1, \ldots, a_m]^\top \in \R^{m \times n}$ is a given data matrix, comprising $m$ data points in $n$ dimensions, and $\mu >0$ is a regularization parameter. The columns of $A$ are all \emph{centered} and \emph{normalized}, so that $\frac{1}{m} A^\top A$ represents the \emph{(biased) covariance matrix} of the data.
}

{
\subsubsection{Real Data} We consider the \emph{Internet Advertisements} dataset, which is publicly available from the UCI Machine Learning Repository \cite{internet_advertisements_51}. This dataset contains $m=3279$ instances and $n=1558$ features. Each row corresponds to an instance, which represents a single image found on a webpage. Each column corresponds to a feature, which encodes the geometry of the image as well as binary indicators for phrases occurring in the image's URL, alt text, anchor text, and surrounding HTML content. Thus, the entry $A_{ij}$ in the data matrix $A$ represents the numerical value of feature $j$ for image $i$; for example, this could be the height or width (in pixels) of the image, or a binary value indicating the presence of a specific keyword. We select $\mu = 0.8$ and $p=10$.}

    


\begin{figure}[ht]
\begin{minipage}[t]{0.5\linewidth} 
\centering 
\includegraphics[scale=0.40]{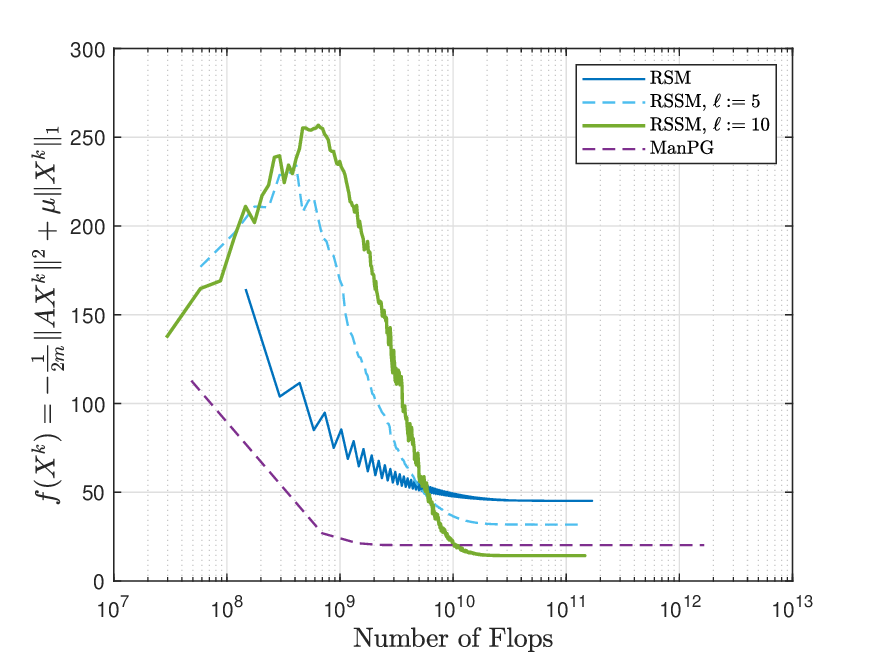} 
\end{minipage}%
\begin{minipage}[t]{0.5\linewidth} 
\centering 
\includegraphics[scale=0.40]{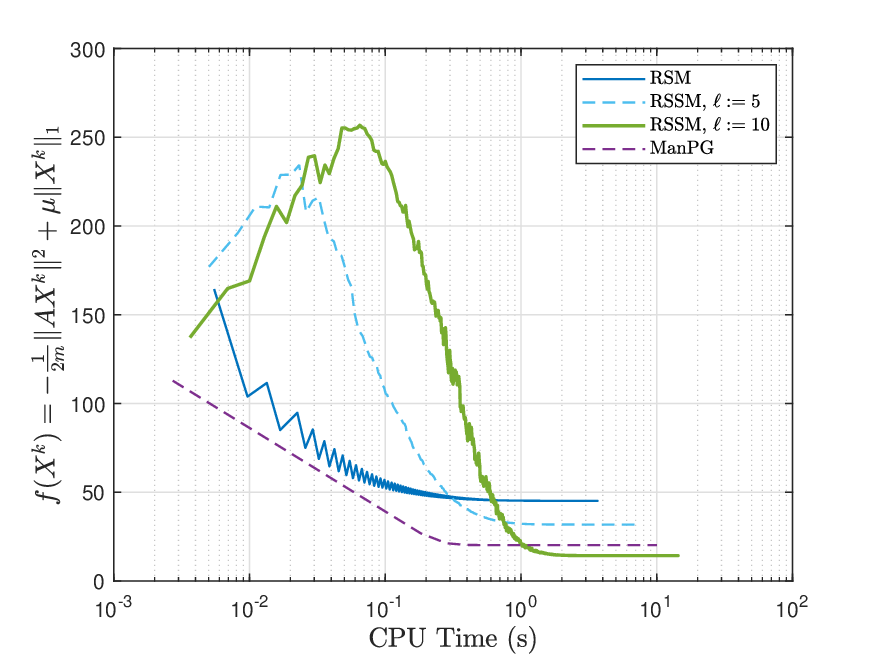}
\end{minipage} 
\vspace{-0.5cm}
\caption{{Sparse PCA: $\gamma_k = \frac{10^{-2} \ell^{20(0.999)^{k\theta}-18}}{\sqrt{k\theta+2}\log(k\theta+2)}$ at $k$-th iteration and $\theta = \frac{2}{\ell}$ for RSSM and RSM, and ManPG with Armijo backtracking line search.}}
\label{fig-spca}
\end{figure}

{
Before running the algorithms, we remove/impute missing entries and
mean-center and normalize each feature column. We compare our proposed RSSM (using $\ell = 5, 10$ blocks) against the full-update RSM~\cite{li2021weakly} with heuristic stepsizes $\gamma_k = \frac{10^{-2} \ell^{20(0.999)^{k\theta}-18}}{\sqrt{k\theta+2}\log(k\theta+2)}$ for $\theta = \frac{2}{\ell}$ ($\ell = 2$ for RSM) and ManPG \cite{chen2020proximal} with Armijo backtracking.}

{
The numerical results in Figure~\ref{fig-spca} demonstrate that our proposed RSSM achieves the lowest objective value among all the tested algorithms. Specifically, RSSM with $\ell = 10$ reaches an objective value of $14.2864$, which substantially improves upon that obtained by ManPG. In fact, ManPG stagnates at the higher value of $20.2402$.}

{
While ManPG exhibits a fast initial descent---reaching the value of $20.2403$ in just $0.2441$ seconds---this speed is theoretically expected. As a proximal gradient method, it computes a descent direction in the tangent space at the current iterate by solving a quadratic subproblem related to the proximal map of the $1$-norm in Problem~\eqref{eq-sPCA}, which can be done very efficiently. However, RSSM demonstrates highly competitive computational efficiency alongside its better solution quality: RSSM ($\ell = 10$) surpasses the best objective value returned by ManPG in just $1.0181$ seconds (reaching $19.9798$) and continues to decrease to its final reported value of $14.2864$ in $4.8233$ seconds.}

{
Furthermore, RSSM consistently and significantly outperforms the baseline full-update RSM, which validates the computational efficiency of our randomized block-coordinate strategy. Finally, it is important to emphasize that while ManPG is highly effective for optimization problems where certain proximal map yields easily solvable quadratic subproblems, RSSM and RSM are pure subgradient methods. Thus, RSSM offers a broader framework that applies to general weakly convex nonsmooth objective functions.}

\section{Conclusion}

In this paper, we proposed a new coordinate-type algorithm RSSM for solving nonsmooth weakly convex optimization problems over high dimensional Stiefel manifolds.
The main idea of RSSM is that at each iteration, it decomposes the Stiefel manifold into submanifold blocks and performs a retracted partial Riemannian subgradient step with respect to a randomly selected submanifold block. {When the cost of computing a partial Euclidean subgradient is proportional to the block size,} RSSM enjoys a low per-iteration cost and is especially suitable for high-dimensional applications. Furthermore, we showed that RSSM converges to the set of stationary points at a sublinear rate. To the best of our knowledge, RSSM is the first feasible, coordinate-type algorithm with convergence guarantees for nonsmooth weakly convex optimization over Stiefel manifolds.




\bibliographystyle{siamplain}
\bibliography{references}

@ARTICLE{boumal2019global,
 author={Nicolas Boumal and P.-A. Absil and Coralia Cartis},
 title={Global Rates of Convergence for Nonconvex Optimization on Manifolds},
 journal = {IMA Journal of Numerical Analysis},
 volume  = {39},
 number  = {1},
 pages   = {1--33},
 year={2019}
}

@article{chen2024nonsmooth,
    author = {Shixiang Chen and Shiqian Ma and Anthony Man-Cho So and Tong Zhang},
    title = {Nonsmooth Optimization over the {S}tiefel Manifold and Beyond: Proximal Gradient Method and Recent Variants},
    journal = {SIAM Review},
    volume={66},
    number={2},
    pages={319--352},
    year = {2024} 
}

@incollection{robbins1971convergence,
author = {H. Robbins and D. Siegmund},
title = {A Convergence Theorem for Non Negative Almost Supermartingales and Some Applications},
editor = {Jagdish S. Rustagi},
booktitle = {Optimizing Methods in Statistics},
pages = {233-257},
year = {1971},
publisher={Academic Press}
}

@book{billingsley1995probability,
    author = {Patrick Billingsley},
    title = {Probability and Measure},
    publisher = {John Wiley \& Sons, Inc.},
    edition={Third},
    series={Wiley Series in Probability and Mathematical Statistics},
    year = {1995}
}

@ARTICLE{absil2012projection,
 author={P.-A.~Absil and J\'{e}r\^{o}me Malick},
 title={Projection-Like Retractions on Matrix Manifolds},
 journal={SIAM Journal on Optimization},
 volume={22},
 number={1},
 pages={135--158},
 year={2012}
}

@article{absil2019collection,
  title={{A collection of nonsmooth Riemannian optimization problems}},
  author={Absil, P-A and Hosseini, Seyedehsomayeh},
  journal={Nonsmooth Optimization and Its Applications},
  pages={1--15},
  year={2019},
  publisher={Springer}
}

@book{absil2009optimization,
  author={Absil, P-A and Mahony, Robert and Sepulchre, Rodolphe},
  title={Optimization Algorithms on Matrix Manifolds},
  year={2009},
  publisher={Princeton University Press}
}

@article{adly2016preservation,
  title={Preservation of prox-regularity of sets with applications to constrained optimization},
  author={Adly, Samir and Nacry, Florent and Thibault, Lionel},
  journal={SIAM Journal on Optimization},
  volume={26},
  number={1},
  pages={448--473},
  year={2016},
  publisher={SIAM}
}

@inproceedings{bai2018subgradient,
  title={Subgradient descent learns orthogonal dictionaries},
  author={Bai, Yu and Jiang, Qijia and Sun, Ju},
  booktitle={The 7th International Conference on Learning Representations},
  year={2019}
}

@book{boumal2023introduction,
    author = {Nicolas Boumal},
    title = {An Introduction to Optimization on Smooth Manifolds},
    publisher = {Cambridge University Press},
    year = {2023}
}

@article{chen2020proximal,
  title={Proximal gradient method for nonsmooth optimization over the {Stiefel} manifold},
  author={Chen, Shixiang and Ma, Shiqian and So, Anthony Man-Cho and Zhang, Tong},
  journal={SIAM Journal on Optimization},
  volume={30},
  number={1},
  pages={210--239},
  year={2020},
  publisher={SIAM}
}

@book{clarke1990optimization,
  title={Optimization and Nonsmooth Analysis},
  author={Clarke, Frank H},
  year={1990},
  publisher={SIAM}
}

@article{clarke1995proximal,
  title={Proximal smoothness and the lower-{$C^2$} property},
  author={Clarke, Francis H and Stern, Ronald J and Wolenski, Peter R},
  journal={Journal of Convex Analysis},
  volume={2},
  number={1-2},
  pages={117--144},
  year={1995}
}

@article{cunningham2015linear,
  title={Linear dimensionality reduction: Survey, insights, and generalizations},
  author={Cunningham, John P and Ghahramani, Zoubin},
  journal={The Journal of Machine Learning Research},
  volume={16},
  number={1},
  pages={2859--2900},
  year={2015},
  publisher={JMLR. org}
}

@inproceedings{ding2021dual,
  title={Dual principal component pursuit for robust subspace learning: Theory and algorithms for a holistic approach},
  author={Ding, Tianyu and Zhu, Zhihui and Vidal, Ren{\'e} and Robinson, Daniel P},
  booktitle={Proceedings of the 38th International Conference on Machine Learning},
  pages={2739--2748},
  year={2021}
}

@article{gao2019parallelizable,
  title={Parallelizable algorithms for optimization problems with orthogonality constraints},
  author={Gao, Bin and Liu, Xin and Yuan, Yaxiang},
  journal={SIAM Journal on Scientific Computing},
  volume={41},
  number={3},
  pages={A1949--A1983},
  year={2019},
  publisher={SIAM}
}

@book{golub2013matrix,
  title={Matrix Computations},
  author={Golub, Gene H and Van Loan, Charles F},
  year={2013},
  publisher={The Johns Hopkins University Press},
  edition={Fourth}
}

@book{gower2004procrustes,
  title={Procrustes Problems},
  author={Gower, John C and Dijksterhuis, Garmt B},
  volume={30},
  year={2004},
  series={Oxford Statistical Science Series},
  publisher={Oxford University Press}
}

@book{guillemin2010differential,
  title={Differential Topology},
  author={Guillemin, Victor and Pollack, Alan},
  volume={370},
  year={2010},
  publisher={American Mathematical Society}
}

@article{gutman2022coordinate,
  title={Coordinate Descent Without Coordinates: Tangent Subspace Descent on {Riemannian} Manifolds},
  author={Gutman, David H and Ho-Nguyen, Nam},
  journal={Mathematics of Operations Research},
  year={2022},
  volume={48},
  number={1},
  pages={127--159},
  publisher={INFORMS}
}

@article{hanefficient,
  title={Efficient Optimization with Orthogonality Constraint: a Randomized {R}iemannian Submanifold Method},
  author={Han, Andi and Poirion, Pierre-Louis and Takeda, Akiko},
  journal={arXiv:2505.12378},
  year={2025}
}

@inproceedings{han2024riemannian,
  title={Riemannian coordinate descent algorithms on matrix manifolds},
  author={Han, Andi and Jawanpuria, Pratik and Mishra, Bamdev},
  booktitle={Proceedings of the 41st International Conference on Machine Learning},
  pages={17393--17415},
  year={2024}
}

@inproceedings{huang2021projection,
  title={Projection robust {Wasserstein} barycenters},
  author={Huang, Minhui and Ma, Shiqian and Lai, Lifeng},
  booktitle={Proceedings of the 38th International Conference on Machine Learning},
  pages={4456--4465},
  year={2021}
}

@inproceedings{huang2021riemannian,
  title={A {Riemannian} block coordinate descent method for computing the projection robust {Wasserstein} distance},
  author={Huang, Minhui and Ma, Shiqian and Lai, Lifeng},
  booktitle={Proceedings of the 38th International Conference on Machine Learning},
  pages={4446--4455},
  year={2021}
}

@inproceedings{jiang2022givens,
  title={Givens coordinate descent methods for rotation matrix learning in trainable embedding indexes},
  author={Jiang, Yunjiang and Zhang, Han and Qiu, Yiming and Xiao, Yun and Long, Bo and Yang, Wen-Yun},
  booktitle={The 10th International Conference on Learning Representations},
  year={2022}
}

@article{journee2010generalized,
  title={Generalized power method for sparse principal component analysis.},
  author={Journ{\'e}e, Michel and Nesterov, Yurii and Richt{\'a}rik, Peter and Sepulchre, Rodolphe},
  journal={Journal of Machine Learning Research},
  volume={11},
  number={2},
  year={2010}
}

@article{li2019orthogonal,
  title={Orthogonal deep neural networks},
  author={Li, Shuai and Jia, Kui and Wen, Yuxin and Liu, Tongliang and Tao, Dacheng},
  journal={IEEE Transactions on Pattern Analysis and Machine Intelligence},
  volume={43},
  number={4},
  pages={1352--1368},
  year={2019},
  publisher={IEEE}
}

@article{li2021weakly,
  title={Weakly convex optimization over {Stiefel} manifold using {Riemannian} subgradient-type methods},
  author={Li, Xiao and Chen, Shixiang and Deng, Zengde and Qu, Qing and Zhu, Zhihui and So, Anthony Man-Cho},
  journal={SIAM Journal on Optimization},
  volume={31},
  number={3},
  pages={1605--1634},
  year={2021},
  publisher={SIAM}
}

@article{liu2015analysis,
  title={{On the analysis of the discretized Kohn--Sham density functional theory}},
  author={Liu, Xin and Wen, Zaiwen and Wang, Xiao and Ulbrich, Michael and Yuan, Yaxiang},
  journal={SIAM Journal on Numerical Analysis},
  volume={53},
  number={4},
  pages={1758--1785},
  year={2015},
  publisher={SIAM}
}

@article{liu2019quadratic,
  title={{Quadratic optimization with orthogonality constraint: Explicit {\L}ojasiewicz exponent and linear convergence of retraction-based line-search and stochastic variance-reduced gradient methods}},
  author={Liu, Huikang and So, Anthony Man-Cho and Wu, Weijie},
  journal={Mathematical Programming},
  volume={178},
  pages={215--262},
  year={2019},
  publisher={Springer}
}

@inproceedings{massart2022coordinate,
  title={Coordinate descent on the orthogonal group for recurrent neural network training},
  author={Massart, Estelle and Abrol, Vinayak},
  booktitle={Proceedings of the AAAI Conference on Artificial Intelligence},
  volume={36},
  pages={7744--7751},
  year={2022}
}

@misc{internet_advertisements_51,
  author       = {Kushmerick, Nicholas},
  title        = {{Internet Advertisements}},
  year         = {1999},
  howpublished = {UCI Machine Learning Repository},
  note         = {{DOI}: https://doi.org/10.24432/C5V011}
}

@article{nesterov2012efficiency,
  title={Efficiency of coordinate descent methods on huge-scale optimization problems},
  author={Nesterov, Yu},
  journal={SIAM Journal on Optimization},
  volume={22},
  number={2},
  pages={341--362},
  year={2012},
  publisher={SIAM}
}

@article{peng2023block,
  title={Block coordinate descent on smooth manifolds},
  author={Peng, Liangzu and Vidal, Ren{\'e}},
  journal={arXiv:2305.14744},
  year={2023}
}

@article{darmwal2023low,
  title={Low-complexity subspace-descent over symmetric positive definite manifold},
  author={Darmwal, Yogesh and Rajawat, Ketan},
  journal={arXiv:2305.02041},
  year={2023}
}

@book{rockafellar2009variational,
  title={Variational Analysis},
  author={Rockafellar, R Tyrrell and Wets, Roger J-B},
  volume={317},
  year={2009},
  publisher={Springer Science \& Business Media}
}

@article{sato2013riemannian,
  title={{A Riemannian optimization approach to the matrix singular value decomposition}},
  author={Sato, Hiroyuki and Iwai, Toshihiro},
  journal={SIAM Journal on Optimization},
  volume={23},
  number={1},
  pages={188--212},
  year={2013},
  publisher={SIAM}
}

@inproceedings{shalit2014coordinate,
  title={Coordinate-descent for learning orthogonal matrices through {G}ivens rotations},
  author={Shalit, Uri and Chechik, Gal},
  booktitle={Proceedings of the 31st International Conference on Machine Learning},
  pages={548--556},
  year={2014}
}

@inproceedings{theis2009soft,
  title={{Soft dimension reduction for ICA by joint diagonalization on the Stiefel manifold}},
  author={Theis, Fabian J and Cason, Thomas P and Absil, P -A},
  booktitle={Proceedings of the 8th International Conference on Independent Component Analysis and Signal Separation},
  pages={354--361},
  year={2009}
}

@article{tian2021distributed,
  title={Distributed certifiably correct pose-graph optimization},
  author={Tian, Yulun and Khosoussi, Kasra and Rosen, David M and How, Jonathan P},
  journal={IEEE Transactions on Robotics},
  volume={37},
  number={6},
  pages={2137--2156},
  year={2021},
  publisher={IEEE}
}

@article{vial1983strong,
  title={Strong and weak convexity of sets and functions},
  author={Vial, Jean-Philippe},
  journal={Mathematics of Operations Research},
  volume={8},
  number={2},
  pages={231--259},
  year={1983},
  publisher={INFORMS}
}

@inproceedings{vu2022distributionally,
  title={Distributionally robust fair principal components via geodesic descents},
  author={Vu, Hieu and Tran, Toan and Yue, Man-Chung and Nguyen, Viet Anh},
  booktitle={The 10th International Conference on Learning Representations},
  year={2022}
}

@article{wang2023decentralized,
  title={Decentralized Weakly Convex Optimization Over the {Stiefel} Manifold},
  author={Wang, Jinxin and Hu, Jiang and Chen, Shixiang and So, Anthony Man-Cho},
  journal={IEEE Transactions on Signal Processing},
  volume={74},
  pages={3404--3418},
  year={2026}
}

@article{wright2015coordinate,
  title={Coordinate descent algorithms},
  author={Wright, Stephen J},
  journal={Mathematical Programming},
  volume={151},
  number={1},
  pages={3--34},
  year={2015},
  publisher={Springer}
}

@article{yang2014optimality,
  title={Optimality conditions for the nonlinear programming problems on {Riemannian} manifolds},
  author={Yang, Wei Hong and Zhang, Lei-Hong and Song, Ruyi},
  journal={Pacific Journal of Optimization},
  volume={10},
  number={2},
  pages={415--434},
  year={2014}
}

@article{zhao2022randomized,
  title={Randomized Coordinate Subgradient Method for Nonsmooth Optimization},
  author={Zhao, Lei and Chen, Ding and Zhu, Daoli and Li, Xiao},
  journal={arXiv:2206.14981},
  year={2022}
}

@article{zhu2018dual,
  title={Dual principal component pursuit: Improved analysis and efficient algorithms},
  author={Zhu, Zhihui and Wang, Yifan and Robinson, Daniel and Naiman, Daniel and Vidal, Rene and Tsakiris, Manolis},
  journal={Advances in Neural Information Processing Systems},
  volume={31},
  year={2018}
}

@article{cheung2024randomized,
  title={Randomized Submanifold Subgradient Method for Optimization over {S}tiefel Manifolds},
  author={Cheung, Andy Yat-Ming and Wang, Jinxin and Yue, Man-Chung and So, Anthony Man-Cho},
  journal={arXiv:2409.01770},
  year={2024}
}

@article{hosseini2013nonsmooth,
  title={Nonsmooth optimization techniques on {R}iemannian manifolds},
  author={Hosseini, S and Pouryayevali, MR},
  journal={Journal of Optimization Theory and Applications},
  volume={158},
  number={2},
  pages={328--342},
  year={2013},
  publisher={Springer}
}

\appendix

\section{Technical Lemmas}

\begin{lemma}
\label{lem-commutative}
Let $X \in \Stiefel n p$ be fixed and $\mathfrak C$ be a partition of $[p]$ with $\ell = |\mathfrak C| \geq 2$. Then, for any $\xi \in \R^{n\times p}$, we have
\begin{align}
    \label{eq-commutative} 
    &\textstyle\mathcal P_{T_X \Stiefel n p} \circ \mathcal A_X(\xi) = \mathcal A_X \circ \mathcal P_{T_X\Stiefel n p}(\xi) = \frac{1}{\binom{\ell}{2}}\,\sum\limits_{\{i,j\} \in \binom{[\ell]}{2}}\mathcal P_{T_{X_{ij}}\mathcal M_{X_{-ij}}}(\xi_{ij})I_{ij}^\top,\\
    \label{eq-commutative1}
    &\textstyle \mathcal P_{T_{X_{ij}}\mathcal M_{X_{-ij}}}\left(\mathcal A_X(\xi)I_{ij}\right)I_{ij}^\top = \mathcal A_X\left(\mathcal P_{T_{X_{ij}}\mathcal M_{X_{-ij}}}(\xi_{ij})I_{ij}^\top\right).
\end{align}
\end{lemma}

\begin{proof} Let $\xi \in \R^{n \times p}$ be fixed. On one hand, a direct computation gives
\begin{align*}
& \mathcal A_{X}\left(\mathcal P_{T_X \Stiefel n p}(\xi)\right) \\
&\textstyle= \frac{1}{\binom{\ell}{2}} \sum\limits_{\{i,j\} \in \binom{[\ell]}{2}} (I-X_{-ij}X_{-ij}^\top)\left(X\Skew{X^\top\xi} + (I-XX^\top) \xi\right) I_{ij} I_{ij}^\top\\
&\textstyle= \frac{1}{\binom{\ell}{2}} \sum\limits_{\{i,j\} \in \binom{[\ell]}{2}} \left(X_{ij} I_{ij}^\top \Skew{X^\top \xi} I_{ij} + (I-XX^\top)\xi_{ij}\right)I_{ij}^\top\\
&\textstyle
= \frac{1}{\binom{\ell}{2}}\,\sum\limits_{\{i,j\} \in \binom{[\ell]}{2}}\mathcal P_{T_{X_{ij}}\mathcal M_{X_{-ij}}}(\xi_{ij})I_{ij}^\top.
\end{align*}
On the other hand, using~\eqref{eq-scaling-coordrep} and the fact that $XX^\top + X_\perp X_\perp^\top = I$, we have $\mathcal A_X(\xi) = X \left(\binom{\ell}{2}^{-1} Q \boxdot (X^\top \xi)\right) + (I-XX^\top)(\frac{2}{\ell}\,\xi)$. This gives
\begin{align*}
\mathcal P_{T_X \Stiefel n p}\left(\mathcal A_X(\xi)\right)
&\textstyle= X \Skew{ \binom{\ell}{2}^{-1} Q \boxdot (X^\top \xi) } + (I-XX^\top)(\frac{2}{\ell}\,\xi)\\
&\textstyle= X\,\left(\binom{\ell}{2}^{-1} Q \boxdot \Skew{X^\top \xi} \right) + (I-XX^\top)(\frac{2}{\ell}\,\xi)\\ 
&\textstyle= \mathcal A_X \left( \mathcal P_{T_X\Stiefel n p}(\xi) \right).
\end{align*}
Hence, \eqref{eq-commutative} holds. Next, we compute 
\begin{align*}
    &\textstyle \mathcal P_{T_{X_{ij}}\mathcal M_{X_{-ij}}}\left(\mathcal A_X(\xi)I_{ij}\right)I_{ij}^\top\\ 
    &\textstyle= \mathcal P_{T_{X_{ij}}\mathcal M_{X_{-ij}}}\left(\left(X \left(\binom{\ell}{2}^{-1}\,Q \boxdot (X^\top \xi)\right) + X_\perp (\frac{2}{\ell} X_\perp^\top \xi) \right)I_{ij}\right)I_{ij}^\top\\
    &\textstyle= X \left(\binom{\ell}{2}^{-1} Q \boxdot (I_{ij} \Skew{X_{ij}^\top\xi_{ij}} I_{ij}^\top)\right) + (I-XX^\top)(\frac{2}{\ell}\,\xi_{ij})I_{ij}^\top,\\ 
    &\textstyle
    \mathcal A_X\!\left(\mathcal P_{T_{X_{ij}}\mathcal M_{X_{-ij}}}(\xi_{ij})I_{ij}^\top\right)\\
    &\textstyle= \mathcal A_X \left(X (I_{ij} \Skew{X_{ij}^\top\xi_{ij})I_{ij}^\top} + X_\perp (X_\perp^\top \xi_{ij}I_{ij}^\top)\right)\\
    &\textstyle= X \left(\binom{\ell}{2}^{-1} Q \boxdot (I_{ij} \Skew{X_{ij}^\top\xi_{ij})I_{ij}^\top} \right) + (I-XX^\top)(\frac{2}{\ell}\,\xi_{ij})I_{ij}^\top.
\end{align*}
It follows that~\eqref{eq-commutative1} holds.
\end{proof}

\begin{lemma} \label{lem-selfadjoint}
Let $X \in \Stiefel n p$ be fixed and $\mathfrak C$ be a partition of $[p]$ with $\ell = |\mathfrak C| \geq 2$. The averaging operator $\mathcal A_X$ is self-adjoint and has eigenvalues $\frac{2}{\ell}$ and $\binom{\ell}{2}^{-1}$ with multiplicities $(n-p)p+\sum_{i=1}^\ell p_i^2$ and $p^2 - \sum_{i=1}^\ell p_i^2$, respectively. 
\end{lemma}

\begin{proof}
Fix any $X_\perp \in \Stiefel{n}{n-p}$ such that $[X\ X_\perp] \in \Stiefel{n}{n}$. Then, for any $\xi, \eta \in \R^{n \times p}$, we have
\begin{align*}
    \langle\,\mathcal A_X (\xi), \eta\,\rangle 
    &\textstyle= \left\langle\, X \left(\binom{\ell}{2}^{-1}\,Q \boxdot (X^\top \xi)\right) + X_\perp (\frac{2}{\ell}\,X_\perp^\top\xi), X (X^\top \eta) + X_\perp(X_\perp^\top \eta)\,\right\rangle\\
    &\textstyle= \left\langle\, \binom{\ell}{2}^{-1}\,Q \boxdot (X^\top \xi), X^\top \eta\,\right\rangle + \left\langle\, \frac{2}{\ell}\,X_\perp^\top\xi, X_\perp^\top\eta\,\right\rangle\\
    &\textstyle= \left\langle\, X^\top \xi, \binom{\ell}{2}^{-1}\,Q \boxdot (X^\top \eta)\,\right\rangle + \left\langle\, X_\perp^\top\xi, \frac{2}{\ell}\,X_\perp^\top\eta\,\right\rangle
    = \langle\,\xi, \mathcal A_X(\eta)\,\rangle,
\end{align*}
i.e., $\mathcal A_X$ is self-adjoint. In addition, the following subspaces
\begin{align*}
    &\left\{\,X \left[\begin{smallmatrix}
    A_{11} & 0 & \cdots & 0\\
    0 & A_{22} & & 0\\
    \vdots & & \ddots & \vdots\\
    0 & 0 & \cdots & A_{\ell \ell}
    \end{smallmatrix}\right] + X_\perp B \,:\, A_{ii} \in \R^{p_i \times p_i}, B \in \R^{(n-p) \times p} \,\right\},\\
    &\left\{\,X \left[\begin{smallmatrix}
    0 & A_{12} & \cdots & A_{1\ell}\\
    A_{21} & 0 & & A_{2\ell}\\
    \vdots & & \ddots & \vdots\\
    A_{\ell1} & A_{\ell2} & \cdots & 0
    \end{smallmatrix}\right] \,:\, A_{ij} \in \R^{p_i \times p_j}\,\right\}
\end{align*}
are the eigenspaces of $\mathcal A_X$ corresponding to the eigenvalues $\frac{2}{\ell}$ and $\binom{\ell}{2}^{-1}$, respectively.
\end{proof}

\begin{lemma} \label{lem-diffscaling}
Let $X, Y \in \Stiefel n p$ be fixed. For any $\xi \in \R^{n \times p}$, we have
\begin{align} \label{eq-diffscaling1}
\textstyle\mathcal A_{Y}^{-1}(\xi) - \mathcal A_{X}^{-1}(\xi) = \frac{\ell(\ell-2)}{2} \left( Y\left[\left(J-I\right) \boxdot (Y^\top \xi)\right] - X\left[\left(J-I\right) \boxdot (X^\top \xi)\right] \right).
\end{align}
\end{lemma}

\begin{proof}
Let $Q'=J-\frac{\ell-2}{\ell-1}I$ denote the entrywise reciprocal of $Q=J+(\ell-2)I$. It can be verified that
\[ \textstyle\mathcal A_{X}^{-1}(\xi) = X\left(\binom{\ell}{2}\,Q' \boxdot (X^\top \xi)\right) + (I-XX^\top)(\frac{\ell}{2}\,\xi) \]
for all $\xi \in \R^{n \times p}$. It follows that
\begin{align*}
&\mathcal A_{Y}^{-1}(\xi) - \mathcal A_{X}^{-1}(\xi)\\
&\textstyle= Y\left(\binom{\ell}{2}\, Q' \boxdot (Y{}^\top \xi)\right) + (I-YY^\top)(\frac{\ell}{2}\,\xi)\\ 
&\textstyle\qquad -\, X\left(\binom{\ell}{2}\,Q' \boxdot (X^\top \xi)\right) - (I-XX^\top)(\frac{\ell}{2}\,\xi)\\
&\textstyle= Y\left(\binom{\ell}{2}\, Q' \boxdot (Y^\top \xi)\right) - Y(\frac{\ell}{2}\,Y^\top\xi) - X\left(\binom{\ell}{2}\,Q' \boxdot (X^\top \xi)\right) + X(\frac{\ell}{2}\,X^\top\xi)\\
&\textstyle= \frac{\ell}{2}\,\left( Y\left[\left((\ell-1) Q' - J\right) \boxdot (Y^\top \xi)\right] - X\left[\left((\ell-1) Q' - J\right) \boxdot (X^\top \xi)\right] \right)\\
&\textstyle= \frac{\ell(\ell-2)}{2} \left( Y\left[\left(J-I\right) \boxdot (Y^\top \xi)\right] - X\left[\left(J-I\right) \boxdot (X^\top \xi)\right] \right),
\end{align*}
which completes the proof.
\end{proof}

\begin{lemma} \label{lem-almostNonexpansive2}
Consider the setting of Lemma~\ref{lem-almostNonexpansive}. For any $\xi, \zeta \in \Stiefel n p$, we have
\begin{align}
    \textstyle \left|\left(\|\xi\|_{\mathcal A_{X^{+}}^{-1}}^2 - \|\xi\|_{\mathcal A_{X}^{-1}}^2\right) - \left(\|\zeta\|_{\mathcal A_{X^{+}}^{-1}}^2 - \|\zeta\|_{\mathcal A_{X}^{-1}}^2\right)\right| \leq {\sqrt{2}}{\ell(\ell-2)}\, \|X^{+}-X\| \cdot \left\|\xi -\zeta\right\|. \label{eq-almostNonexpansive2}
\end{align}
\end{lemma}

\begin{proof}
Recall that $\Psi_X(\xi) = X\left[\left(J-I\right) \boxdot (X^\top \xi)\right]$. By \eqref{eq-diffscaling1} and \eqref{eq-diffPsi}, we have
\begin{align*}
&\textstyle\left(\|\xi\|_{\mathcal A_{X^{+}}^{-1}}^2 - \|\xi\|_{\mathcal A_{X}^{-1}}^2\right) - \left(\|\zeta\|_{\mathcal A_{X^{+}}^{-1}}^2 - \|\zeta\|_{\mathcal A_{X}^{-1}}^2\right)\\
&\textstyle= \frac{\ell(\ell-2)}{2}\left(\langle\,(\Psi_{X^{+}}-\Psi_X)(\xi), \xi\,\rangle - \langle\,(\Psi_{X^{+}}-\Psi_X)(\zeta), \zeta\,\rangle\right)\\
&\textstyle= \frac{\ell(\ell-2)}{2}\left(\langle\,(X_i^{+}X_i^{+}{}^\top - X_iX_i^\top)\xi_{-i}I_{-i}^\top + (X_j^{+}X_j^{+}{}^\top - X_j X_j^\top)\xi_{-j}I_{-j}^\top, \xi\,\rangle\right.\\
&\qquad\textstyle\qquad\ \left. -\, \langle\,(X_i^{+}X_i^{+}{}^\top - X_iX_i^\top)\zeta_{-i}I_{-i}^\top + (X_j^{+}X_j^{+}{}^\top - X_j X_j^\top)\zeta_{-j}I_{-j}^\top, \zeta\,\rangle\right)\\
&\textstyle= \frac{\ell(\ell-2)}{2}\left(\langle\,X_i^{+}X_i^{+}{}^\top - X_iX_i^\top, \xi_{-i}\xi_{-i}^\top - \zeta_{-i}\zeta_{-i}^\top\,\rangle\right.\\
&\qquad\textstyle\qquad\ \left. +\, \langle\,X_j^{+}X_j^{+}{}^\top - X_jX_j^\top, \xi_{-j}\xi_{-j}^\top - \zeta_{-j}\zeta_{-j}^\top\,\rangle\right)\\
&\textstyle \leq \frac{\ell(\ell-2)}{2} \left(\|X_{i}^{+}X_{i}^{+}{}^\top - X_{i}X_{i}^\top\| \cdot \|\xi_{-i}\xi_{-i}^\top - \zeta_{-i}\zeta_{-i}^\top\|\right.\\
&\qquad\textstyle\qquad\ \left. +\, \|X_{j}^{+}X_{j}^{+}{}^\top - X_{j}X_{j}^\top\| \cdot \|\xi_{-j}\xi_{-j}^\top - \zeta_{-j}\zeta_{-j}^\top\|\right)\\
&\textstyle \leq \frac{\ell(\ell-2)}{2} \left( {2}\|X_i^{+}-X_i\| \cdot \|\xi_{-i}-\zeta_{-i}\| + {2}\|X_j^{+}-X_j\| \cdot \|\xi_{-j}-\zeta_{-j}\|\right)\\ 
&\textstyle \leq {\sqrt{2}\ell(\ell-2)}\|X_{ij}^{+}-X_{ij}\| \cdot \|\xi-\zeta\|.
\end{align*}
{
Here, 
the second inequality is due to the fact that for any
$U,V \in \Stiefel{n}{q}$ with $q \in [n]$, we have
\begin{align*}
\|UU^\top-VV^\top\|^2
&= 2q-2\|U^\top V\|^2 = 2\|(I-UU^\top)V\|^2 \\
&= 2\|(I-UU^\top)(V-U)\|^2 \leq 2\|V-U\|^2.
\end{align*}
To get the last inequality, we write $\delta=\xi-\zeta$ and note that
\(
\|\delta_{-i}\|^2+\|\delta_{-j}\|^2
=
\|\delta_i\|^2+\|\delta_j\|^2
+2\|\delta_{-ij}\|^2
\leq 2\|\delta\|^2.
\)
Consequently, the Cauchy--Schwarz inequality yields
\begin{align*}
&\quad \|X_i^+-X_i\| \cdot \|\delta_{-i}\|
 +\|X_j^+-X_j\| \cdot \|\delta_{-j}\| \\
&\leq
\left(\|X_i^+-X_i\|^2+\|X_j^+-X_j\|^2\right)^{1/2}
\left(\|\delta_{-i}\|^2+\|\delta_{-j}\|^2\right)^{1/2} \leq
\sqrt{2}\,\|X_{ij}^+-X_{ij}\| \cdot \|\delta\|.
\end{align*}
Now, since $X_{-ij}^+=X_{-ij}$, we have
$\|X_{ij}^+-X_{ij}\|=\|X^+-X\|$.
It follows that
\[ \left(\|\xi\|_{\mathcal A_{X^{+}}^{-1}}^2 - \|\xi\|_{\mathcal A_{X}^{-1}}^2\right) - \left(\|\zeta\|_{\mathcal A_{X^{+}}^{-1}}^2 - \|\zeta\|_{\mathcal A_{X}^{-1}}^2\right) \le \sqrt{2} \ell(\ell-2)\, \|X^+-X\| \cdot \|\xi-\zeta\|. \]
By repeating the above derivation with $\xi$ and $\zeta$ interchanged, we see that the same upper bound holds. This establishes~\eqref{eq-almostNonexpansive2}, as desired. 
}
\end{proof}

\begin{lemma}[Convergence Lemma; {see~\cite[Lemma 4.1]{zhao2022randomized}}] \label{convlem}
Let $\{Y^k\}\subseteq \R^{n \times p}$ and $\{\mu_k\}\subseteq \R_{+}$ be given sequences. Let $\Phi : \R^{n\times p} \to \R^m$ be a map and $L_\Phi > 0$ be a constant such that $\| \Phi (Y^k) - \Phi(Y^{k+1}) \| \le L_{\Phi} \|Y^k - Y^{k+1}\|$ for all $k\ge 0$. Suppose there exist constants $M, r > 0$ and a vector $\bar\Phi \in \R^m$ such that
\begin{itemize}
\item $\|Y^k - Y^{k+1}\| \leq M\mu_k$ for all $k \ge 0$,
\item $\sum_{k\ge0} \mu_k = \infty$, 
\item $\sum_{k\ge0} \mu_k \|\Phi(Y^k) - \bar\Phi\|^{r} < \infty$.
\end{itemize}
Then, we have $\lim_{k \to \infty} \| \Phi(Y^k) - \bar{\Phi}\|^{r} = 0$. 
\end{lemma}

\end{document}